\documentclass[english,reqno]{amsart}

\usepackage{amsmath, appendix, ulem}
\usepackage{mathtools}  
\mathtoolsset{showonlyrefs} 
\usepackage[shortlabels]{enumitem}
\usepackage[nobysame]{amsrefs}
\usepackage{amssymb, color}
\usepackage[margin=1in]{geometry}
\usepackage{mathrsfs}
\usepackage{graphicx}
\usepackage{subfig}
\usepackage{float}
\usepackage{epsf}
\usepackage{hyperref}
\usepackage{titletoc}

\usepackage{subeqnarray}
\usepackage{cases}
\graphicspath{{../Figures/}}

\numberwithin{equation}{section}

\newtheorem{lem}{Lemma}[section]
\newtheorem{thm}{Theorem}[section]
\newtheorem{example}{Example}[section]
\newtheorem{proposition}[thm]{Proposition}

\theoremstyle{remark}
\newtheorem{rmk}{Remark}[section]
\newtheorem{definition}[thm]{Definition}
\newtheorem{assum}{Assumption}

\newcommand{\nn}{\nonumber}
\newcommand{\R}{{\mathbb R}}

\renewcommand{\tilde}{\widetilde}

\renewcommand{\bar}{\overline}

\newcommand{\mc}[1]{\mathcal{#1}}

\newcommand{\EE}{\mathbb{E}}
\newcommand{\RR}{\mathbb{R}}

\newcommand{\PP}{\mathbb{P}}

\newcommand{\OV}{\overline{V}}
\newcommand{\OX}{\overline{X}}
\newcommand{\OY}{\overline{Y}}

\newcommand{\bxy}{\bar y(x)}
\newcommand{\bxyn}{\bar y(\OX_0)}
\newcommand{\bxynt}{\bar y(\OX_t)}

\newcommand{\TE}{\mathcal{E}}
\newcommand{\CX}{\mathcal{X}}
\newcommand{\CY}{\mathcal{Y}}

\DeclareMathOperator\diag{diag}

\author{Giacomo Borghi}
\address{Department of Mathematics, Heriot-Watt University, Edinburgh, United Kingdom}
\email{g.borghi@hw.ac.uk}

\author{Hui Huang}
\address{Department of Mathematics and Scientific Computing, Karl-Franzens-Universit\"{a}t Graz, Universit\"{a}tspl. 3, 8010 Graz, Austria}
\email{hui.huang@uni-graz.at}

\author{Jinniao Qiu}
\address{Department of Mathematics and Statistics, University of Calgary, 2500 University Dr NW, Calgary, AB T2N 1N4, Canada}
\email{jinniao.qiu@ucalgary.ca}

\begin{document}
\title[CBO for MinMax problems]{A Particle Consensus Approach to solving  nonconvex-nonconcave min-max problems}
%{A consensus-based algorithm for nonconvex-nonconcave  Min-Max problems}
\maketitle
\begin{abstract}
We propose a zero-order optimization method for sequential min-max problems based on two populations of interacting particles. The systems are coupled so that one population aims to solve the inner maximization problem, while the other aims to solve the outer minimization problem. The dynamics are characterized by a consensus-type interaction with additional stochasticity to promote exploration of the objective landscape.
Without relying on convexity or concavity assumptions, we establish theoretical convergence guarantees of the algorithm via a suitable mean-field approximation of the particle systems. Numerical experiments illustrate the validity of the proposed approach. In particular, the algorithm is able to identify a global min-max solution, in contrast to gradient-based methods, which typically converge to possibly suboptimal stationary points.
\end{abstract}
% (G.: I deleted saddle point, as that is not the probelm we aim to solve) 
\section{Introduction}

\subsection{Motivation}
%\subsection{Sequential min-max problems}
%\subsection{}

The min-max optimization problem, also known as the saddle point problem, represents a compelling yet complex category of challenges. The objective here is to determine a value for the argument that minimizes the objective function's value even in the worst-case scenario within a particular class of objective functions. This approach was initially introduced in the realm of multi-player zero-sum games in game theory \cite{von2007theory}. Such problems are widespread in various domains, including mathematics, biology, social sciences, and especially in economics \cite{myerson1997game}. In recent years, min-max optimization has also seen broad applications in signal processing, such as distributed processing \cite{chang2020distributed}, robust transceiver design, fair resource allocation \cite{liu2013max}, and communication strategies in the presence of jammers \cite{gohary2009generalized}. Moreover, this optimization technique has gained prominence in modern machine learning contexts, including generative adversarial networks (GANs) \cite{goodfellow2020generative}, fair machine learning \cite{madras2018learning}, adversarial training \cite{madry2018towards}, and multi-agent reinforcement learning \cite{omidshafiei2017deep}. For an overview of further applications and challenges, see \cite{razaviyayn2020nonconvex}.

To initiate our discussion, consider the scenario involving two agents. The first agent aims to minimize the objective function $\TE~: \mc{X}\times \mc{Y}\to \RR$, whereas the second agent's goal is to maximize it, namely
\begin{equation}\label{minmaxprobl}
	\min_{x\in  \mc{X}\subseteq \R^{d_1}}\max_{y\in \mc{Y} \subseteq\R^{d_2}}\TE(x,y).
\end{equation}
In this context, $x$ and $y$ denote the optimization variables, with $\mc{X}$ and $\mc{Y}$ representing feasible sets. The objective function $\TE$ may exhibit non-convexity with respect to $x$ and non-concavity with respect to $y$. Specifically, within the framework of GANs, $x$ symbolizes the parameters of a generator, whereas $y$ corresponds to those of a discriminator. Here, $\TE(\cdot, \cdot)$ serves as the cost function for the generator,  that is equal to the negative of the discriminator's cost function.  In adversarial training, $x$ is the weight of the neural network that one aims to train while $y$ models possible adversarial attacks, and $\TE(\cdot, \cdot)$ the loss function.

When min-max problem is defined in a symmetrical way, namely
\begin{equation}\label{symcon}
	\min_{x\in  \mc{X}}\max_{y\in \mc{Y} }\TE(x,y)=\max_{y\in \mc{Y} }\min_{x\in  \mc{X}}\TE(x,y)\,,
\end{equation}
 a well-known notion of optimality originating from game theory is the one of Nash equilibria (also referred to as saddle points)~\cite{nash1950equilibrium}, where neither of the players has anything to gain by changing only his own strategy.
 This concept is formalized within the following definition \cite{zhang2022optimality}:
 \begin{definition} \label{defnash} 
 	A point $(x^*,y^*)\in\CX\times\CY$ is called Nash equilibrium or saddle point of a function $\TE$ if it holds
 	\begin{equation*}
 		\TE(x^*,y)\leq \TE(x^*,y^*)\leq \TE(x,y^*) \quad\text{for all } (x,y)\in \CX\times\CY\,.
 	\end{equation*}
 In other words, we have simultaneously:
 \begin{equation*}
 	x^*\in \arg\min_{x\in \CX}\TE(x,y^*),\quad  	y^*\in \arg\max_{y\in \CY}\TE(x^*,y)\,.
 \end{equation*}
 \end{definition}
 According to \cite{jensen2001robust}, the symmetrical property \eqref{symcon} is a sufficient and necessary condition for the existence of Nash equilibrium $(x^*,y^*)$ in closed convex domains. This is corresponding to simultaneous games, where each player chooses its action without the knowledge of the action chosen by the other player, so both players act simultaneously. While in sequential games, there is an intrinsic order ($\min_x\max_y$ and $\max_y\min_x$) that players take their actions. Actually, GANs and adversarial training are in fact sequential games in their standard formulations, where Nash equilibria may not exist \cites{farnia2020gans}.  For example, in the formulation of GANs, one first find the optimal parameters of the discriminator $y$ based on the parameters of the generator $x$, and then optimize over $x$. 
% 
% In a two-player zero-sum game, the common payoff function $\TE(x,y)$ represents the gain of $y$-player as well as the loss of $x$-player, where we may call $x$ the min-player and $y$ the max-player. In simultaneous games, each player chooses its action without the knowledge of the action chosen by the other player, so both players act simultaneously. While in sequential games, there is an intrinsic order ($\min_x\max_y$ and $\max_y\min_x$) that players take their actions. Actually, GANs and adversarial training are in fact sequential games in their standard formulations. However in the classical case where the payoff function $\TE$ is \textit{convex-concave} (i.e. $\TE(\cdot,y)$ is convex for all $y$, and $\TE(x,\cdot)$ is concave for all $x$), the intrinsic order of the sequential games does not matter under an additional compactness assumption on either $\mc{X}$ or  $\mc{Y}$ by the well-known minmax theorem \cite{v1928theorie,sion1958general}.
In this case, we first introduce the upper  envelope functions:
\begin{equation}
	\overline \TE(x)= \max_{y\in \CY}\TE(x,y)\,,%\quad \mbox{ and }\quad \underline\TE(x)= \min_{x\in \R^{d_1}}\TE(x,y)\,.
\end{equation}
which is allowed to take value $+\infty$.
Then it  leads to the following solution concept \cite{zhang2022optimality}:
\begin{definition}\label{defsolminmax}
	Point $(x^*,y^*)\in\CX\times \CY$ is a global min-max solution to  \eqref{minmaxprobl} if
	\begin{equation}
	x^*\in \arg\min_{x\in \CX}\overline\TE(x)\quad \mbox{ and }\quad y^*=y^*(x^*)\in\arg\max_{y\in \CY}\TE(x^*,y)\,,
	\end{equation}
in other words,  for all $(x,y)\in \CX\times \CY$ it holds
	\begin{equation}
	\TE(x^*,y)\leq \TE(x^*,y^*)=\overline \TE(x^*)\leq \overline \TE(x)\,.
	\end{equation}
\end{definition}
\noindent Similarly one can use the lower envelope function $\underline \TE(y):=\min_{x\in\CX}\TE(x,y)$ to define a global max-min solution to a max-min problem $\max_{y\in \mc{Y} }\min_{x\in  \mc{X}}\TE(x,y)$.

\begin{figure}
\includegraphics[width = 1\linewidth]{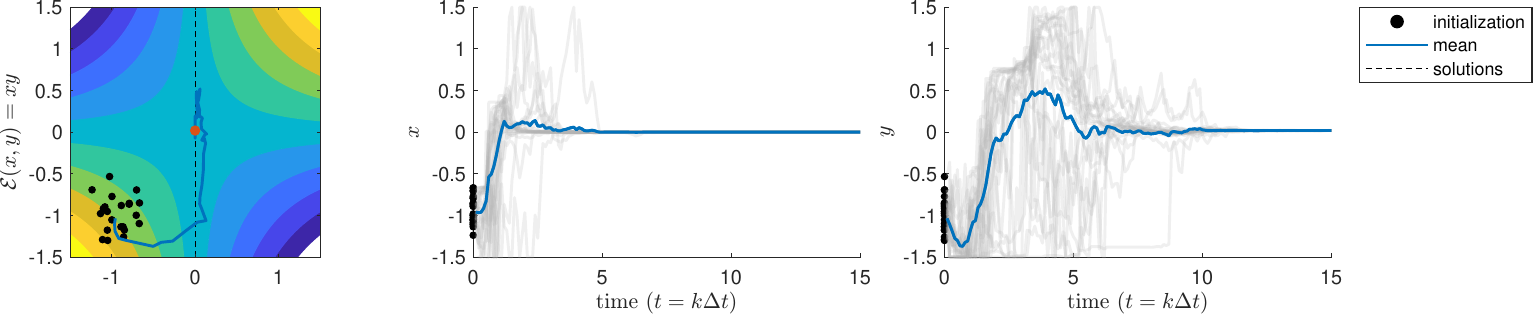}
\caption{Illustration of a run of the proposed consensus-based algorithm to solve Example \ref{ex_bilinear}). We employ $N=25$ particles, both in $\mathcal{X}$ and $\mathcal{Y}$, which all reach consensus around the global min-max solution global solution $(0,0)$. Refer to \cite{hsieh2021limits} for information on how the mean of particles does exhibit the typical cycling behavior of gradient-based  methods. 
}
\label{fig:illustrative}
\end{figure}

To better understand the definitions, we introduce the following examples. The first example considers a simple bilinear (hence convex-concave) objective function:
\begin{example} \label{ex_bilinear}
	Consider the bilinear objective function 
	\begin{equation*}
		\TE(x,y):=xy,\quad  (x,y)\in \R\times \R\,.
	\end{equation*} 
	It is easy to verify that global min-max points are $\{0\}\times \R$ while global max-min points are $\R\times \{0\}$. Taking the intersection we obtain the unique global saddle point $(0,0)$.
\end{example}
\noindent In the next example, the objective function $\TE$ possesses both global min-max and max-min points, yet it lacks a saddle point.
\begin{example} \label{ex_bivariate}
	Consider the bivariate function
	\begin{equation*}
		\TE(x,y)=x^4/4-x^2/2+xy,\quad (x,y)\in \R\times \R\,.
	\end{equation*}
	One can check that global min-max points occur at $\{0\}\times \R$ with a value $0$, whereas the global max-min points are located at $(\pm1,0)$ with values of $-1/4$. Nonetheless, there are no saddle points.
	\end{example}

Although the concept of global min-max is clearly defined in Definition \ref{defsolminmax}, it encounters significant challenges when applied to nonconvex-nonconcave scenarios. The primary issue is the lack of an efficient algorithm for finding a global minimizer $x^*$ of a non-convex function $\overline \TE(x)$. This challenge can be partially addressed by aiming for a local minimizer or a stationary point. Moreover, once $x^*$ is determined, finding a global maximizer $y^*$ for the non-concave function $\TE(x^*,y)$ is NP-hard. While it may be tempting to settle for a local solution, doing so compromises the initial notion of optimality for $x^*$. Therefore, there is a pressing need for a new method to effectively address these complexities arising from nonconvex-nonconcave problems. In contrast with \cite{huang2022consensus} where a consensus-based method is introduced for finding global saddle points satisfying Definition \ref{defnash}, this paper aims to propose a new method for finding global min-max solutions that only meet the criteria of Definition \ref{defsolminmax}. This method is especially relevant where the condition of symmetry, as in \eqref{symcon}, may not hold, as illustrated by Example \ref{ex_bivariate}.

\subsection{Literature review}

Due to the aforementioned applications in Machine Learning, there is a renewed interest in the development of solvers for nonconvex-nonconcave minmax problems. The Gradient Descent Ascent (GDA) algorithm is a first order algorithms which couples a descent dynamics in the $x$ direction, with an ascent one in $y$ one. This can be done simultaneously (single-loop algorithm), or by alternating the two steps (double-loops). While being suitable for convex-concave problems \cite{cheru2017saddle}, GDA fails to convergence even in simple problems like example \ref{ex_bilinear}, see e.g. \cite{chau2019negative}. More generally, in \cites{hsieh2021limits,letcher2021imposs}, the authors investigate the convergence properties of gradient-based methods in the nonconvex-nonconcave settings proving existence of cycles or attractors that are not global min-max solutions in the sense of Definition \ref{defsolminmax}. Many improved variants of GDA have been proposed and analyzed to mitigate this effects see, for example, \cites{chau2019negative,zheng2023universal,wang2020ridge} to name a few. We note that gradient-based methods aim to converge to stationary solutions which are not necessary global minmax solutions, unless the objective function satisfies additional assumptions such as the two-sided Polyak--Łojasiewicz inequality condition \cite{yang2020global}. For additional details on the relation between (local) minimax solutions, stationary points, and Nash equilibria, we refer to \cite{jin2020what}.

Alternative to GDA-like dynamics, gradient-free (zero-order) methods aim to solve \eqref{minmaxprobl} by only relying on the function evaluations. This is, for instance, the scenario where the attacker, in the GAN framework, has only access to the model's output.
This can be done, for example, via a randomized gradient estimation \cites{liu2020zero}, model approximations of the objective function \cites{meni2020derivativefree}, or evolutonary algorithms \cites{cramer2009ea,herrmann1999ga,laskari2002pso,chen2015pso}. The last class comprises metaheuristic startegies \cites{Blum:2003:MCO:937503.937505,Gendreau:2010:HM:1941310} where a set of particles explore the search space via deterministic and random interactions. Authors in \cites{laskari2002pso,chen2015pso}, in particular, adapt the particle-swarm evolution \cites{kennedy1995particle}, while in \cites{cramer2009ea,herrmann1999ga} the genetic algorithm paradigm \cites{Holland:1992:ANA:531075} is used. We note that, despite their popularity in the field of non-smooth optimization, such evolutionary methods typically lack convergence guarantees.

Our approach will be inspired by the framework of consensus-based optimization (CBO), initially introduced by \cites{carrillo2018analytical,PTTM}. CBO belongs to the family of global optimization methodologies, which leverages systems of interacting particles to achieve consensus around global minimizers of the cost functions. As part of the broader class of metaheuristics, %\cites{Blum:2003:MCO:937503.937505,Gendreau:2010:HM:1941310}, 
CBO orchestrates interactions between local improvement procedures and global strategies, utilizing both deterministic and stochastic processes. This interplay ultimately results in an efficient and robust procedure for exploring the solution space of the cost functions.
A distinctive advantage of the CBO approach is its gradient-free nature, making it particularly appealing for optimizing cost functions that are either non-smooth or where the computation of derivatives is overly burdensome.

In response to a wide range of applications, the foundational CBO model has been extended and refined to accommodate diverse scenarios. These enhancements include the incorporation of memory aspects or gradient information \cites{riedl2022leveraging,totzeck2020consensus,cipriani2022zero}, and the integration of momentum mechanisms \cite{chen2022consensus}. Moreover, CBO has been adapted to address challenges in global optimization on compact manifolds \cites{fornasier2020consensus,ha2022stochastic}, to tackle general constraints \cites{borghi2023constrained,carrillo2023consensus}, to optimize functions with multiple minimizers \cite{bungert2022polarized,fornasier2024pde}, to solve multi-objective problems \cites{borghi2022consensus,borghi2022adaptive}, and to sample from specific distributions \cite{carrillo2022consensus}. Further applications of CBO have spanned high-dimensional machine learning issues \cites{carrillo2021consensus,fornasier2021consensus}, asset allocation tasks \cite{bae2022constrained}, clustered federated learning \cite{carrillo2023fedcbo}, and more recently, non-convex multiplayer games \cite{chenchene2023consensus}. Finally, during the preparation of this work, we observed that the authors in \cite{herty2024multiscale}—before their work was posted online—independently developed a CBO dynamic similar to ours for bi-level optimization problems. However, their focus was primarily on the numerical aspects, employing a multi-time scaling approach without providing proof of global convergence.

\subsection{Contribution}

In this paper, we propose a zero-order, consensus-based method aimed at identifying global min-max solutions that comply with Definition \ref{defsolminmax}. The algorithm employs two sets of interacting particles, $N$ in the $x$ direction, and $N$ in the $y$ direction, which are coupled together to specifically solve the sequential min-max problem \eqref{minmaxprobl}.
The particle system is described via a system of $2N$-dimensional SDEs. 

To study the large-time behavior of the dynamics, we propose a mean-field approximation that reduces the system to only two mean-field SDEs of Mckean--Vlasov type. Inspired by \cite{fornasier2021consensus1} we establish convergence towards global min-max problems under specific assumptions on the objective function. We assume, in particular, the uniqueness of the solution and the growth conditions around it for the objective function. Our analysis does not require any convexity/concavity assumption, nor even differentiability of $\TE$.

We test the proposed algorithm on benchmark min-max problems which are nonconvex-nonconcave, namely the "Bilinearly-coupled" \cite{grimmer2023prox}, the 
"Forsaken" \cite{hsieh2021limits}, and the "Sixth-order polynomial" problem \cite{wang2020ridge}. While state-of-the-art gradient-based methods converges to sub-optimal stationary points \cite{zheng2023universal}, the proposed CBO method successfully computes optimal solutions in the sense of Definition \ref{defsolminmax}, see Figure \ref{fig:illustrative_2}. Some numerical tests are also conducted to investigate how different parameters of the algorithms affect the convergence properties.

The rest of the paper is organized as follows. In Section \ref{sec:dynamics} we illustrate the particle system and its mean-field approximation. Section \ref{sec:convergence} is devoted to the theoretical analysis of the mean-field dynamics and its convergence properties. %Results on benchmark problems are presented in Section \ref{sec:numerics}, while final remarks are left in Section \ref{sec:conclusions}. Some technical aspects of the proofs are presented in the Appendix.
Results on benchmark problems are presented in Section \ref{sec:numerics}, while some technical aspects of the proofs are presented in the Appendix.

\begin{figure}
\centering
\includegraphics[trim = 0cm 0cm 2.8cm 0cm,clip, width = 0.17\linewidth]{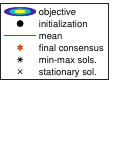}
\subfloat[\label{fig:ill_3}]{\includegraphics[trim = 1.9cm 0cm 0cm 0cm,clip,height = 0.25\linewidth]{fig_illustrative_bis_3}}\hfill
\subfloat[\label{fig:ill_4}]{\includegraphics[trim = 1.9cm 0cm 0cm 0cm,clip,height = 0.25\linewidth]{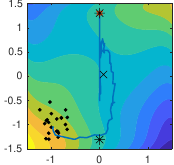}}\hfill
\subfloat[\label{fig:ill_10}]{\includegraphics[trim = 1.9cm 0cm 0cm 0cm,clip,height = 0.25\linewidth]{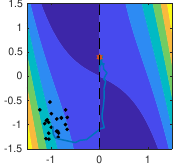}}
\caption{Illustration of a run of the proposed consensus-based algorithm for the "Forsaken" (A), "Sixth-order polynomial" (B), and Example \ref{ex_bivariate} (C) problems. Unlike gradient-based methods (see, for instance, \cite[Figure 4]{zheng2023universal}) the particles create consensus around global min-max solutions rather than stationary points. The dotted line in (C) indicated the set of min-max solutions. Problems definition for (A), (B) are given in Section \ref{sec:numerics}}
\label{fig:illustrative_2}
\end{figure}

\section{The particle dynamics}
\label{sec:dynamics}

For simplicity, we take $\mathcal X=\mathcal Y=\mathbb R^d$ as the extensions to cases with $\mathcal X$ and $\mathcal Y$ having different dimensions are straightforward. Specifically, our CBO dynamic for min-max problem is a collection of $2N$ $ \R^{d}\times \R^{d}$-valued particles $\{(X_t^{i},Y_t^{i})\}_{i\in [N]}$ satisfying the following Stochastic Differential Equations (SDEs)
\begin{subequations}\label{particle}
	\begin{numcases}{}
		dX_t^i=-\lambda_1(X_t^i-X_{\alpha,\beta}(\rho_t^{N,X}))dt+\sigma_1 D(X_t^i-	X_{\alpha,\beta}(\rho_t^{N,X}))dB_t^{X,i}\,, \quad i\in[N]:=\{1,\dots,N\}\,,\label{eqX}\\
		dY_t^i=-\lambda_2(Y_t^i-Y_\beta(\rho_t^{N,Y},X_t^i))dt+\sigma_2 D(Y_t^i-Y_\beta(\rho_t^{N,Y},X_t^i))dB_t^{Y,i}\,,\quad i\in[N],
		\label{eqV}
	\end{numcases}
\end{subequations}
where
\begin{equation}\label{YaN}
	Y_\beta(\rho_t^{N,Y},X_t^i)=\frac{\int_{\RR^{d}}y\omega_{-\beta}^{\mc{E}}(X_t^i,y)\rho^{N,Y}(t,dy)}{\int_{\RR^{d}}\omega_{-\beta}^{\mc{E}}(X_t^i,y)\rho^{N,Y}(t,dy)}, \quad \rho^{N,Y}(t,dy)=\frac{1}{N}\sum_{i=1}^N\delta_{Y_t^i}dy\,,
\end{equation}
and 
\begin{equation}\label{XaN}
	X_{\alpha,\beta}(\rho_t^{N,X})=\frac{\int_{\RR^{d}}x\omega_{\alpha}^{\mc{E}}(x,Y_\beta(\rho_t^{N,Y},x))\rho^{N,X}(t,dx)}{\int_{\RR^{d}}\omega_{\alpha}^{\mc{E}}(x,Y_\beta(\rho_t^{N,Y},x))\rho^{N,X}(t,dx)}, \quad \rho^{N,X}(t,dx)=\frac{1}{N}\sum_{i=1}^N\delta_{X_t^i}dx\,.
\end{equation}
Here $\lambda_1,\lambda_2,\sigma_1,\sigma_2>0$ are drift and diffusion parameters, $\{B_t^{m,i}\}_{m=X,Y;i\in [N]}\in \R^d$ are independent standard  Brownian motions, and $D(X):=\diag(X_1,\dots,X_d)$ (anisotropic), or $D(X):= |X|$ (isotropic), for all $X\in \R^d$. However, in the subsequent analysis, we will specifically focus on the anisotropic case, and let $\lambda_1=\lambda_2=\lambda$ and $\sigma_1=\sigma_2=\sigma$ for simplicity.
The system is complemented with identically and independently distributed (i.i.d.) initial data $\{(X_0^{i},Y_0^{i})\}_{i\in [N]}$ with respect to the common law $f_0$.
The first term on the right hand-side of \eqref{particle} is a drift that pulls the particles towards a current consensus point $\left(X_{\alpha,\beta}(\rho_t^{N,X}),Y_\beta(\rho_t^{N,Y},X_t^i)\right)$, in which $\omega_{\alpha}^{\mc{E}}$ is a weight function defined as
\begin{equation}\label{eq:weight_function}
	\omega_{\alpha}^{\mc{E}}(x,y):=\exp\left(-\alpha \TE(x,y)\right)\,, \quad \text{for all } (x,y) \in \R^d\times \R^d\,.
\end{equation}
The introduction of the second stochastic term in \eqref{particle} is intended to promote exploration of the energy landscape of the objective function. When the consensus is achieved, meaning $(X^i,Y^i)=\left(X_{\alpha,\beta}(\rho^{N,X}),Y_\beta(\rho^{N,Y},X^i)\right)$, both the drift and diffusion terms vanish. The choice of the weight function \eqref{eq:weight_function} comes from  the  well-known Laplace's principle \cites{miller2006applied,Dembo2010}, which states that for any probability measure $\mu\in\mc{P}( \RR^d )$, there holds for any fixed $y\in\R^{d}$,
\begin{equation}\label{lap_princ}
	\lim\limits_{\alpha\to\infty}\left(-\frac{1}{\alpha}\log\left(\int_{\RR^{d} }\omega_\alpha^{\TE}(x,y)\mu(d x)\right)\right)=\inf\limits_{x \in \rm{supp }(\mu)} \TE(x,y)\,.
\end{equation}
Then the intuition behind the dynamics stems from Laplace's principle, which implies that for any fixed $x$, it holds that $Y_\beta(\rho^Y,x)\to\arg\max\limits_{y\in \R^d}\TE(x,y)=:\bar y(x)$ as $\beta\to\infty$. This leads to $\overline \TE(x)=\TE(x,\bar y(x))\approx\TE(x,Y_\infty(\rho^Y,x))$.
Meanwhile, we have that $$X_{\alpha,\beta=\infty}(\rho^{X})=\frac{\int_{\RR^{d}}x\omega_{\alpha}^{\mc{E}}(x,Y_\infty(\rho^{Y},x))\rho^{X}(dx)}{\int_{\RR^{d}}\omega_{\alpha}^{\mc{E}}(x,Y_\infty(\rho^{Y},x))\rho^{X}(dx)}\approx\frac{\int_{\RR^{d}}x\omega_{\alpha}^{\mc{E}}(x,\overline y(x))\rho^{X}(dx)}{\int_{\RR^{d}}\omega_{\alpha}^{\mc{E}}(x,\overline y(x))\rho^{X}(dx)}=\frac{\int_{\RR^{d}}x\omega_{\alpha}^{\overline\TE}(x)\rho^{X}(dx)}{\int_{\RR^{d}}\omega_{\alpha}^{\overline\TE}(x)\rho^{X}(dx)}.$$
Applying Laplace's principle once more, one would expect that $X_{\alpha=\infty,\beta=\infty}\approx \arg\min \overline \TE(x)=x^*$ and at the same time $Y_\infty(\rho^Y,x^*)=\bar y(x^*)=y^*$.

The theoretical analysis of the convergence behavior for a CBO method can be approached from two distinct perspectives. The first approach entails a direct examination of the behavior at the microscopic level of individual particles, as detailed in \cites{ha2021convergence}, through an analysis of a particle system with common noise. Conversely, the approach embraced in this paper pivots to an examination at the macroscopic scale, concentrating on the mean-field equation that emerges as the number $N$ of particles approaches infinity. This macroscopic or mean-field approach has been effectively applied in a series of prior studies \cites{fornasier2022anisotropic, fornasier2022convergence, fornasier2021consensus1, huang2023global}.
Indeed, as the number of particles $N\to\infty$, the mean-field limit \cites{huang2022mean,gerber2023mean} result would imply that our CBO particle dynamic \eqref{particle} well approximates solutions of the following  mean-field  kinetic Mckean--Vlasov type equations
\begin{subequations}\label{MVeq}
	\begin{numcases}{}
		d\OX_t=-\lambda(\OX_t-X_{\alpha,\beta}(\rho_t^X))dt+\sigma D(\OX_t-X_{\alpha,\beta}(\rho_t^X))dB_t^X, \label{eqX}\\
		d\OY_t=-\lambda(\OY_t-Y_\beta(\rho_t^Y,\OX_t))dt+\sigma D(\OY_t-Y_\beta(\rho_t^Y,\OX_t))dB_t^Y\,, \label{eqV}
	\end{numcases}
\end{subequations}
where  
\begin{equation}\label{Ya}
	Y_\beta(\rho_t^Y,\OX_t)=\frac{\int_{\RR^{d}}y\omega_{-\beta}^{\mc{E}}(\OX_t,y)\rho^Y(t,dy)}{\int_{\RR^{d}}\omega_{-\beta}^{\mc{E}}(\OX_t,y)\rho^Y(t,dy)}, \quad \rho^Y(t,dy)=\int_{\RR^{d}}f(t,dx,dy)\,,
\end{equation}
and
\begin{equation}\label{Xa}
	X_{\alpha,\beta}(\rho_t^X)=\frac{\int_{\RR^{d}}x\omega_{\alpha}^{\mc{E}}(x,Y_\beta(\rho_t^Y,x))\rho^X(t,dx)}{\int_{\RR^{d}}\omega_{\alpha}^{\mc{E}}(x,Y_\beta(\rho_t^Y,x))\rho^X(t,dx)}, \quad \rho^X(t,dx)=\int_{\RR^{d}}f(t,dx,dy)\,,
\end{equation}
with $f(t,x,v)$ being  the distribution of $(\OX_t,\OV_t)$ at time $t$, which makes the set of equations \eqref{MVeq} nonlinear. 

 The current paper centers on investigating the convergence  of the proposed CBO variant in the context of finding global min-max solutions. The well-posedness results regarding the CBO particle system \eqref{particle} and its mean-field dynamics \eqref{MVeq} have  been omitted from this work. However, it is worth noting that these results closely resemble the well-posedness theorems presented in \cite[Theorem 3, Theorem 6]{huang2022consensus}, where a CBO dynamic is introduced for two player zero-sum games with the goal of finding saddle points satisfying Definition \ref{defnash}.  Moreover, in the following we shall assume the solution to \eqref{MVeq} has the regularity $\sup_{t\in[0,T]}\EE[|\OX_t|^4+|\OY_t|^4]<\infty$ for  any time horizon $T>0$, which can be guaranteed by the assumption on the initial data $\EE[|X_0|^4+|Y_0|^4]<\infty$.

The main objective of this work is to establish the convergence of the mean-field dynamics $(\OX_t,\OY_t)$ in \eqref{MVeq} to the global min-max point $(x^*,y^*)$ as $t$ and $\alpha$ approach infinity. Inspired by \cite{fornasier2021consensus1}, we introduce the variance functions as follows:
\begin{equation}\label{eq:defi_V_m}
	\mc V^X(t):=\EE[|\OX_t-x^*|^2]\,, \text{and} \quad \mc V^Y(t):= \EE[|\OY_t-y^*|^2]\,, \quad \text{for} \ t>0\,.
\end{equation}
By analyzing the decay behavior of the (cumulative) variance function $\mc V(t):=\mc V^X(t) +\mc V^Y(t)$ we establish the convergence of the CBO dynamics to the global min-max point. Specifically, we demonstrate that the function $V(t)$ decays exponentially, achieving any desired level of accuracy $\varepsilon>0$ with a decay rate controllable through the parameters of the CBO method.

%This also implies the convergence of the mean-field PDE \eqref{meanPDE} to a Dirac delta centering at the global NE with respect to the 2-Wasserstein distance, i.e.
%\begin{equation*}
%	W_2(\rho_t,\delta_{(x_1^*,\dots,x_M^*)})\to 0\,,\quad \mbox{as } \alpha, \ t\to \infty\,.
%\end{equation*}

%A direct application of It\^{o}'s formula, the law $f_t:=f(t,\cdot,\cdot)$ at time $t$ is a weak solution to the following  nonlinear Vlasov-Fokker-Plank equation
%\begin{align}\label{PSO eq}
%\partial_{t} f_t&=\sum_{m=1}^{M}\left(\lambda\nabla_{x_m} \cdot\left(\left(x_m-{X}_{\alpha}(\rho_t^m)\right) f_t\right)+\frac{\sigma^{2}}{2} \partial^2_{(x_m)_k(x_m)_k}\left((x_m-{X}_{\alpha}(\rho_t^m))_k^{2} f_t\right)\right)
%\end{align}
%with the initial data $f_0(x_1,\dots,x_M)=\mathcal{L}(\OX_0^1,\dots,\OX_0^M)$.

\section{Global convergence in mean-field law}
\label{sec:convergence}

In this section, we present our main result about the global convergence in mean-field law for objective functions satisfying the following conditions. Our proof methodology draws upon the approaches
found in \cite{fornasier2021consensus1}, yet it introduces a novel perspective
through the application of our new CBO model.
\begin{assum}\label{assum2}
	Throughout this section,  the objective function $\TE \in \mc{C}(\R^{2d})$ satisfies
	\begin{enumerate}[(i)]
		\item there exist some constants $\underline C_\TE$ and $\overline C_\TE$ such that
		\begin{equation}
			\underline C_\TE \leq \TE(x,y)\leq \overline C_{\TE}\quad \mbox{ for any }x,y\in \RR^d\,.
		\end{equation}
		\item there exists some $L_\TE>0$ such that for any $x_1,x_2,y_1,y_2\in \RR^d$ it holds that
		\begin{equation}
			|\TE(x_1,y_1)-\TE(x_2,y_2)|\leq L_{\TE}(1+|x_1|+|x_2|+|y_1|+|y_2|)(|x_1-x_2|+|y_1-y_2|)\,.
		\end{equation}
		\item for any  $x\in \RR^{d}$, there exists some $\bar y(x)$ such that
		\begin{equation}
			\bar y(x)=\arg\max_{y\in\R^{d}}\TE(x,y)\,.
		\end{equation}
		Especially $\bar y(x^*)=y^*$ and $x^*=\arg\min_{x\in \R^{d}}\overline\TE(x)=\arg\min_{x\in \R^{d}}\TE(x,\bar y(x))$.  Moreover there exists some constants $\bar{c}_1,\bar{c}_2>0$ such that
		\begin{align}
			&|\bar y(x_1)-\bar y(x_2)|\leq \bar{c}_1|x_1-x_2|\quad  \mbox{ for any }x_1,x_2\in \RR^{d}\\
			&|\bar y(x)|\leq \bar{c}_2\quad  \mbox{ for any }x \in \RR^{d} \,.
		\end{align}
		\item  there exist $\TE_\infty,\eta,\nu,R_0>0$  such that for all $x\in \RR^{d}$,
		\begin{equation}
			\eta|y-\bar y(x)|\leq |\TE(x,y)-\TE(x,\bar y(x))|^\nu \mbox{ for all }y\in B_{R_0}(\bar y(x))\,,
		\end{equation}
		\begin{equation}
			\TE_\infty<\TE(x,\bar y(x))-\TE(x,y) \mbox{ for all } y\in (B_{R_0}(\bar y(x))^c\,.
		\end{equation}
     \item  for each $q>0$ there exists $r\in (0,R_0]$   such that
		\begin{equation}
			 \sup_{x\in\RR^d}\sup_{y\in B_{r}(\bar y(x))}|\TE(x,y)-\TE(x,\bar y(x))|\leq q  .
		\end{equation}
	
			\item  there exist $\TE_\infty^*,\eta^*>0$  such that
	\begin{equation}
		\eta^*|x-x^*|\leq |\overline \TE(x)-\overline \TE(x^*)|^\nu \mbox{ for all }x\in B_{R_0}(x^*)\,,
	\end{equation}
	\begin{equation}
		\TE_\infty^*<\overline\TE(x)-\overline \TE(x^*) \mbox{ for all } x\in (B_{R_0}(x^*)^c\,,
	\end{equation}
	where $\nu,R_0>0$ are the constants as in (iv).
	\item  for any $q>0$, there exists some $r\in (0,R_0]$ such that
	\begin{equation}
		|\overline\TE(x)-\overline\TE(x^*)|\leq q \mbox{ for all }x\in B_{r}(x^*)\,.
	\end{equation}
	\end{enumerate}
\end{assum}
Note that $B_r(x)$, here and in the sequel, represents the ball centered in $x\in \R^d$ with radius $r>0$ with respect to the $\ell^{\infty}$ norm, i.e., $B_r(x):=\{x' \in \R^d : \max_{k \in [d]}|(x')^k-x^k| \leq r \}$.

\begin{rmk}
In the above assumption, condition $(i)$ states that the objective function is bounded from above and below. This is achievable by considering the function $\tanh(\TE)$ instead, which is bounded and retains the same min-max solution as $\TE$. Condition $(ii)$ requires $\TE$ to be locally Lipschitz. $(iii)$ is tailored to min-max problems, necessitating that the optimized function $\bar y(x)$ be Lipschitz continuous and uniformly bounded, aligning with the maximum theorem \cite[Theorem 3.6]{beavis1990optimisation}. It is noteworthy that our numerical experiments suggest the possibility of relaxing or even eliminating condition $(iii)$. The remaining conditions $(iv)-(vii)$ are standard in the analysis of CBO methods (c.f. \cite{fornasier2021consensus1, carrillo2023fedcbo,chenchene2023consensus,fornasier2022convergence}) and pertain to the tractability of the landscape of $\TE$ around $(x^*, y^*)$ and in the farfield. The assumption does not mandate any convexity-concavity or differentiability requirements; for example, one may check the nonconvex--nonconcave nondifferentiable function $\mathcal E(x,y)=- (y-(1+x^2)^{-1}\sin(x))^2(1+(y-(1+x^2)^{-2}\sin(x))^2)^{-1}
+(1+ \sin^2 x)^{-1}| \tanh x|$ for $(x,y)\in\mathbb R^2$ which has global min-max point $(0,0)$ and satisfies all the conditions (i)--(vii).
\end{rmk}

Next we recall the variance functions
\begin{equation}
\mc{V}^X(t)=\EE[|\OX_t-x^*|^2]\mbox{ and }\mc{V}^Y(t)=\EE[|\OY_t-y^*|^2].
\end{equation}
Then one has the following lemma for a differential inequality of variance functions.
\begin{lem}\label{lemV}
	Let $\TE$ satisfy Assumption \ref{assum2} (i)-(ii), and we assume that $\lambda, \sigma$ satisfy
$2\lambda-\sigma^2>0$. Then the variance function $\mc{V}(t)=\mc{V}^X(t)+\mc{V}^Y(t)$  satisfies
\begin{align}
	\frac{d \mc{V}(t)}{dt}&\leq -(2\lambda-\sigma^2)\mc{V}(t)+2 (\lambda+\sigma^2)(\mc{V}(t))^{\frac{1}{2}}\left(|x^*-X_{\alpha,\beta}(\rho_t^X)|+\EE[|y^*-Y_\beta(\rho_t^Y,\OX_t)|]\right)\nn\\
	&\quad +\sigma^2 \left(|x^*-X_{\alpha,\beta}(\rho_t^X)|^2+\EE[|y^*-Y_\beta(\rho_t^Y,\OX_t)|^2]\right)\,.
\end{align}
\end{lem}
\begin{proof}
Applying It\^o's formula gives that
	\begin{align}
d|\OX_t-x^*|^2=2\lambda(\OX_t-x^*)d\OX_t+\sigma^2|\OX_t-X_{\alpha,\beta}(\rho_t^X)|^2dt\,,
\end{align}
which implies
\begin{align}
\frac{d \mc{V}^X(t)}{dt}&=-2\lambda\EE[(\OX_t-x^*)\cdot(\OX_t-X_{\alpha,\beta}(\rho_t^X))]+\sigma^2\EE[|\OX_t-X_{\alpha,\beta}(\rho_t^X)|^2]\nn\\
&=-2\lambda\mc{V}^X(t)-2\lambda\EE[(\OX_t-x^*)\cdot(x^*-X_{\alpha,\beta}(\rho_t^X))]+\sigma^2\EE[|\OX_t-X_{\alpha,\beta}(\rho_t^X)|^2]\nn\\
&\leq -(2\lambda-\sigma^2)\mc{V}^X(t)+2(\lambda+\sigma^2)(\mc{V}^X(t))^{\frac{1}{2}}|x^*-X_{\alpha,\beta}(\rho_t^X)|+\sigma^2 |x^*-X_{\alpha,\beta}(\rho_t^X)|^2\,.
\end{align}
Similarly one also has
\begin{align}
	\frac{d \mc{V}^Y(t)}{dt}\leq -(2\lambda-\sigma^2)\mc{V}^Y(t)+2(\lambda+\sigma^2)(\mc{V}^Y(t))^{\frac{1}{2}}\EE[|y^*-Y_\beta(\rho_t^Y,\OX_t)|]+\sigma^2 \EE[|y^*-Y_\beta(\rho_t^Y,\OX_t)|^2]\,.
\end{align}
This completes the proof.
\end{proof}

The main idea underpinning our main convergence result consists of showing that
\begin{equation}
	\frac{d \mc{V}(t)}{dt}\leq -\frac{1}{2}(2\lambda-\sigma^2) \mc{V}(t)\,,
\end{equation}
which formally requires that
\begin{align}
	&2(\lambda+\sigma^2)(\mc{V}(t))^{\frac{1}{2}}\left(|x^*-X_{\alpha,\beta}(\rho_t^X)|+\EE[|y^*-Y_\beta(\rho_t^Y,\OX_t)|]\right)\nn\\
 &+\sigma^2 \left(|x^*-X_{\alpha,\beta}(\rho_t^X)|^2+\EE[|y^*-Y_\beta(\rho_t^Y,\OX_t)|^2]\right)\leq\frac{1}{2}(2\lambda-\sigma^2) \mc{V}(t)\,.
\end{align}
This condition is fulfilled by ensuring that $|x^*-X_{\alpha,\beta}(\rho_t^X)|^2+\EE[|y^*-Y_\beta(\rho_t^Y,\OX_t)|^2]$ remains sufficiently small. To establish this, we shall utilize a quantitative estimate of the Laplace principle.

\begin{proposition}\label{propY}
	Let Assumption \ref{assum2} (i)-(v) hold. For each  $x\in \R^{d}$ and $r\in (0,R_0]$,  define  $$\TE_r(x):=\sup_{y\in B_r(\bxy)}\left|\TE(x,y)-\TE(x,\bxy)\right|.$$ Let $0<q\leq \frac{\TE_\infty}{2}$, and define
	$r:=\max\left\{s\in(0,R_0]:~\sup_{x\in\RR^d}\TE_s(x)\leq q\}\right\}$. Then for any fixed $t>0$ it holds that
	\begin{align}\label{318}
		|\bxy-Y_\beta(\rho_t^Y,x)|\leq& \frac{(2q)^\nu}{\eta}+\frac{\exp\left(-\beta q\right)}{ \rho_t^Y (\{y:y\in B_r(\bxy)\})} \int|y-\bxy|\rho_t^Y(dy)\,.
	\end{align}
\end{proposition}
\begin{rmk}
	Such $r>0$ exists because of Assumption \ref{assum2}-$(v)$, and it is independent of $x$. The right hand-side of equation \eqref{318} can become small by selecting a sufficiently small value for $q$ and a suitably large value for $\beta$.
\end{rmk}
\begin{proof}
   Let $\tilde r\geq r>0$, and using Jensen's inequality one can deduce
   \begin{align}
   |\bxy-Y_\beta(\rho_t^Y,x)|&\leq \int_{B_{\tilde r}(\bxy)}|y-\bxy|\frac{\omega_{-\beta}^{\TE}(x,y)\rho_t^Y(dy)}{\|\omega_{-\beta}^{\TE}(x,\cdot)\|_{L_1(\rho_t^Y)}}\nn\\
   	&+\int_{(B_{\tilde r}(\bxy))^c}|x-\bxy|\frac{\omega_{-\beta}^{\TE}(x,y)\rho_t^Y(dy)}{\|\omega_{-\beta}^{\TE}(x,\cdot)\|_{L_1(\rho_t^Y)}}\,.
   \end{align}
The first term is bounded by $\tilde r$ since $|y-\bxy|\leq \tilde r$ for all $y\in B_{\tilde r}(\bxy)$. Moreover, it follows from Markov's inequality that
{\begin{align}
	&\|\omega_{-\beta}^{\TE}(x,\cdot)\|_{L_1(\rho_t^Y)}\nn\\
	\geq& \exp\left(\beta(\TE(x,\bxy)-\TE_r(x))\right)\rho_t^Y\left(\{y:\exp(\beta\TE(x,y))\geq \exp\left(\beta(\TE(x,\bxy)-\TE_r(x))\right)\}\right)\nn\\
=&	\exp\left(\beta(\TE(x,\bxy)-\TE_r(x))\right)\rho_t^Y(\{y: \TE(x,y)\geq \TE(x,\bxy)-\TE_r(x)\})\nn\\
\geq& \exp\left(\beta(\TE(x,\bxy)-\TE_r(x))\right)\rho_t^Y(\{y:y\in B_r(\bxy)\})\,.
\end{align}}
Hence for the second term we have
{ \begin{align}
&\int_{(B_{\tilde r}(\bxy))^c}|y-\bxy|\frac{\omega_{-\beta}^{\TE}(x,y)\rho_t^Y(dy)}{\|\omega_{-\beta}^{\TE}(x,\cdot)\|_{L_1(\rho_t^Y)}}\nn\\
\leq&\frac{1}{\exp\left(\beta(\TE(x,\bxy)-\TE_r(x))\right)\rho_t^Y(\{y:y\in B_r(\bxy)\})} \int_{(B_{\tilde r}(\bxy))^c}|y-\bxy|\omega_{-\beta}^{\TE}(x,y)\rho_t^Y(dy)\nn\\
\leq&\frac{\exp(\beta\sup_{y\in (B_{\tilde r}(\bxy))^c}\TE(x,y))}{\exp\left(\beta(\TE(x,\bxy)-\TE_r(x))\right)\rho_t^Y(\{y:y\in B_r(\bxy)\})} \int|y-\bxy|\rho_t^Y(dy)\nn\\
=&\frac{\exp\left(\beta(\sup_{y\in (B_{\tilde r}(\bxy))^c}\TE(x,y)+\TE_r(x)-\TE(x,\bxy))\right)}{ \rho_t^Y (\{y:y\in B_r(\bxy)\})} \int|y-\bxy|\rho_t^Y(dy)\,.
\end{align}}
Thus for any $\tilde r\geq r>0$ we obtain
{ \begin{align}\label{14}
	&|\bxy-Y_\beta(\rho_t^Y,x)|\nn\\
	\leq& \tilde r+\frac{\exp\left(\beta(\sup_{y\in (B_{\tilde r}(\bxy))^c}\TE(x,y)+\TE_r(x)-\TE(x,\bxy))\right)}{ \rho_t^Y (\{y:y\in B_r(\bxy)\})} \int|y-\bxy|\rho_t^Y(dy)\,.
\end{align}}
Next we choose $\tilde r=\frac{(q+\TE_r(x))^\nu}{\eta}$. Furthermore it holds that
\begin{align}
	\tilde r&=\frac{(q+\TE_r(x))^\nu}{\eta}\geq \frac{\TE_r(x)^\nu}{\eta}=\frac{\left(\sup_{y\in B_r(\bxy)}\left|\TE(x,y)-\TE(x,\bxy)\right|\right)^\nu}{\eta}\nn\\
	&\geq \sup_{y\in B_r(\bxy)}|y-\bxy|=r
\end{align}
by using Assumption \ref{assum2}-$(iv)$. Additionally one notices that
\begin{align}
	&\sup_{y\in (B_{\tilde r}(\bxy))^c}\TE(x,y)-\TE(x,\bxy)\nn\\
 \leq&
 \begin{cases}
    - (\tilde r\eta)^{\frac{1}{\nu}}, \mbox{ when } y\in (B_{\tilde r}(\bxy))^c \cap B_{R_0}(\bxy)\\
      -\TE_\infty, \mbox{ when } y\in (B_{\tilde r}(\bxy))^c \cap (B_{R_0}(\bxy))^c
 \end{cases}
	\leq - (\tilde r\eta)^{\frac{1}{\nu}}\,,
\end{align}
where in the last inequality we have used the fact that $ (\tilde r\eta)^{\frac{1}{\nu}}=q+\TE_r(x)\leq 2q\leq \TE_\infty$.
Therefore
\begin{align}
	\sup_{y\in (B_{\tilde r}(\bxy))^c}\TE(x,y)+\TE_r(x)-\TE(x,\bxy)&\leq - (\tilde r\eta)^{\frac{1}{\nu}}+\TE_r(x)\nn\\
	&= -q-\TE_r(x)+\TE_r(x)=-q\,.
\end{align}
Inserting this and the definition of $\tilde r$ into \eqref{14}, we obtain the desired result.
\end{proof}

\begin{proposition}\label{propX}
	Let $\TE$ satisfy Assumption \ref{assum2}. For any  $r\in (0,R_0]$,  define  $$\TE_r:=\sup_{x\in B_r(x^*)}\left|\overline \TE(x)-\overline \TE(x^*)\right|.$$ Let $0<q^*\leq \frac{\TE_\infty^*}{2}$, and define
	$r:=\max\left\{s\in(0,R_0]:~\TE_s\leq q^*\}\right\}$, then for any  $t\in[0,T]$ it holds
	\begin{align}
		|x^*-X_{\alpha,\beta}(\rho_t^X)|&\leq\frac{(2q^*)^\nu}{\eta^*}+\frac{\exp\left(-\alpha q^*\right)}{ \rho_t^X (\{x:x\in B_r(x^*)\})} \int|x-x^*|\rho_t^X(dx)\nn\\
		&+C_1^\alpha \left(\int_{\RR^{d}}|\bxy-Y_\beta(\rho_t^Y,x)|^2\rho^X(t,dx)\right)^{1/2}
		\,,
	\end{align}
where $C_1^\alpha>0$ depends only on $\alpha, \bar C_{\TE}, \underline C_{\TE},L_{\TE},\bar{c}_1, \EE[|\OX_0|^4],T,\lambda$ and $\sigma$.
\end{proposition}

\begin{proof}
	First, let us introduce
	\begin{equation}
		\OX_{\alpha}(\rho_t^X)=\frac{\int_{\RR^{d}}x\omega_{\alpha}^{\mc{E}}(x,\bar y(x))\rho^X(t,dx)}{\int_{\RR^{d}}\omega_{\alpha}^{\mc{E}}(x,\bar y(x))\rho^X(t,dx)}=\frac{\int_{\RR^{d}}x\omega_{\alpha}^{\overline{\mc{E}}}(x)\rho^X(t,dx)}{\int_{\RR^{d}}\omega_{\alpha}^{\overline{\mc{E}}}(x)\rho^X(t,dx)}\,.
	\end{equation}
Then we split the error
\begin{equation}\label{24}
		|x^*-X_{\alpha,\beta}(\rho_t^X)|\leq |x^*-\OX_{\alpha}(\rho_t^X)|+|\OX_{\alpha}(\rho_t^X)-X_{\alpha,\beta}(\rho_t^X)|\,.
\end{equation}
Similar to the proof of Proposition \ref{propY} it is easy to get
	\begin{align}\label{25}
	|x^*-\OX_{\alpha}(\rho_t^X)|\leq& \frac{(2q^*)^\nu}{\eta^*}+\frac{\exp\left(-\alpha q^*\right)}{ \rho_t^X (\{x:x\in B_r(x^*)\})} \int|x-x^*|\rho_t^X(dx)\,.
\end{align}
As for $|\OX_{\alpha}(\rho_t^X)-X_{\alpha,\beta}(\rho_t^X)|$ we compute
\begin{align}
	&|\OX_{\alpha}(\rho_t^X)-X_{\alpha,\beta}(\rho_t^X)|
        =\left|\frac{\int_{\RR^{d}}x\omega_{\alpha}^{\mc{E}}(x,\bar y(x))\rho^X(t,dx)}{\int_{\RR^{d}}\omega_{\alpha}^{\mc{E}}(x,\bar y(x))\rho^X(t,dx)}-\frac{\int_{\RR^{d}}x\omega_{\alpha}^{\mc{E}}(x,Y_\beta(\rho_t^Y,x))\rho^X(t,dx)}{\int_{\RR^{d}}\omega_{\alpha}^{\mc{E}}(x,Y_\beta(\rho_t^Y,x))\rho^X(t,dx)}\right|
 \nn\\
 &=
 \left|\frac{\int_{\RR^{d}}x\omega_{\alpha}^{\mc{E}}(x,\bar y(x))\rho^X(t,dx)\cdot 
 \int_{\RR^{d}}\big( \omega_{\alpha}^{\mc{E}}(x,Y_\beta(\rho_t^Y,x))-\omega_{\alpha}^{\mc{E}}(x,\bar y(x))\big)\rho^X(t,dx)}{\int_{\RR^{d}}\omega_{\alpha}^{\mc{E}}(x,\bar y(x))\rho^X(t,dx)\int_{\RR^{d}}\omega_{\alpha}^{\mc{E}}(x,Y_\beta(\rho_t^Y,x))\rho^X(t,dx)}\right.\nn\\
    &\quad\quad 
        \left. +\frac{\int_{\RR^{d}}x\big( \omega_{\alpha}^{\mc{E}}(x,\bar y(x))-     \omega_{\alpha}^{\mc{E}}(x,Y_\beta(\rho_t^Y,x))\big)\rho^X(t,dx)}
                                            {\int_{\RR^{d}}\omega_{\alpha}^{\mc{E}}(x,Y_\beta(\rho_t^Y,x))\rho^X(t,dx)}\right|
 \nn\\
	&\leq \alpha e^{2\alpha(\bar C_{\TE}-\underline C_{\TE})}\| x\rho_t^X\|_1\int_{\RR^{d}}|\TE(x,\bar y(x))-\TE(x,Y_\beta(\rho_t^Y,x))|\rho^X(t,dx)\nn\\
	&\quad 
            +\alpha e^{\alpha(\bar C_{\TE}-\underline C_{\TE})}\int_{\RR^{d}}|x||\TE(x,\bar y(x))-\TE(x,Y_\beta(\rho_t^Y,x))|\rho^X(t,dx)\nn\\
        &\leq \alpha e^{2\alpha(\bar C_{\TE}-\underline C_{\TE})}\| x\rho_t^X\|_1\int_{\RR^{d}}L_{\TE} (1+ |x|+|\bar y (0)|+ \bar{c}_1|x|) |\bxy-Y_\beta(\rho_t^Y,x)|\rho^X(t,dx)\nn\\
	&\quad 
            +\alpha e^{\alpha(\bar C_{\TE}-\underline C_{\TE})}\int_{\RR^{d}}|x|
            L_{\TE} (1+ |x|+|\bar y (0)|+ \bar{c}_1|x|) |\bxy-Y_\beta(\rho_t^Y,x)|
            \rho^X(t,dx)\nn\\
	&\leq C(\alpha,\bar C_{\TE},\underline C_{\TE},L_{\TE},\bar{c}_1,\EE[|\OX_t|^4])\left(\int_{\RR^{d}}|\bxy-Y_\beta(\rho_t^Y,x)|^2\rho^X(t,dx) \right)^{1/2}
	\,.
\end{align}

Applying It\^o formula to \eqref{eqX} and taking expectations on both sides, we have
\begin{align}
	\frac{d}{dt}\EE[|\OX_t|^4]
        &\leq -4\lambda\EE[|\OX_t|^{2}\OX_t\cdot(\OX_t-X_{\alpha,\beta}(\rho_t^X))]
        +6\sigma^2\EE[|\OX_t|^{2}|\OX_t-X_{\alpha,\beta}(\rho_t^X)|^2]\nn\\
	&\leq C (\EE[|\OX_t|^4]+\EE[|X_{\alpha,\beta}(\rho_t^X)|^4])\nn\\
	&\leq C(1+e^{4\alpha(\bar C_{\TE}-\underline C_{\TE})})\EE[|\OX_t|^4]\,,
\end{align}
where one has used the fact that $\EE[|X_{\alpha,\beta}(\rho_t^X)|^4]=|X_{\alpha,\beta}(\rho_t^X)|^4\leq e^{4\alpha(\bar C_{\TE}-\underline C_{\TE})}\EE[|\OX_t|^4]$. Applying Gr{\"o}nwall's inequality yields that
\begin{equation}
	\EE[|\OX_t|^4]\leq \EE[|\OX_0|^4]\exp(Ct)\,,
\end{equation}
where $C$ depends only on $\alpha, \bar C_{\TE}, \underline C_{\TE},\sigma$ and $\lambda$. Then we have
\begin{align}
	&|\OX_{\alpha}(\rho_t^X)-X_{\alpha,\beta}(\rho_t^X)|\nn\\
	\leq& C_1^\alpha\left(\int_{\RR^{d}}|\bxy-Y_\beta(\rho_t^Y,x)|^2\rho^X(t,dx)\right)^{1/2}
	\!\!,
\end{align}
$C_1^\alpha>0$ depends only on $\alpha, \bar C_{\TE}, \underline C_{\TE},L_{\TE},\bar{c}_1, \EE[|\OX_0|^4], T, \lambda$ and $\sigma$.
This together with \eqref{24}-\eqref{25} completes the proof.
\end{proof}
To eventually apply the above propositions, one needs to ensure that $\rho_t^X (\{x:x\in B_r(x^*)\})$  and $\rho_t^Y (\{y:y\in B_r(\bxy)\})$ are bounded away from $0$ for a finite time horizon $T$. To do so, we employ a rather technical argument inspired from \cite[Proposition 23]{fornasier2021consensus1}, and for any $|\bar y|\leq \bar c_2$ introduce the mollifier $\phi_r^{\bar y}\colon \R^{2d}\to\R$  for $r>0$ defined by
	\begin{align}\label{moli}
	\phi_{r}^{\bar y}(x,y):= \begin{cases}\prod_{k=1}^{d} \exp \left(1-\frac{r^{2}}{r^{2}-\left(x-x^{*}\right)_{k}^{2}}\right)\prod_{k=1}^{d} \exp \left(1-\frac{r^{2}}{r^{2}-\left(y-\bar y\right)_{k}^{2}}\right), & \text { if } (x,y)\in B_r(x^*)\otimes B_r(\bar y) \\ 0, & \text { else }.\end{cases}
	\end{align}
	We have $\mbox{Image}(\phi_r^{\bar y}) = [0, 1], \mbox{supp}(\phi_r^{\bar y}) = B_r(x^*)\otimes B_r(\bar y),  \phi_r^{\bar y}\in \mc{C}_c^\infty(\RR^{2d})$ and
	\begin{align}
	\begin{aligned}
	\partial_{x_k} \phi_{r}^{\bar y} &=-2 r^{2} \frac{\left(x-x^{*}\right)_{k}}{\left(r^{2}-\left(x-x^{*}\right)_{k}^{2}\right)^{2}} \phi_{r}^{\bar y}, \\
	\partial_{x_k x_k}^{2} \phi_{r}^{\bar y}&=2 r^{2}\left(\frac{2\left(2\left(x-x^{*}\right)_{k}^{2}-r^{2}\right)\left(x-x^{*}\right)_{k}^{2}-\left(r^{2}-\left(x-x^{*}\right)_{k}^{2}\right)^{2}}{\left(r^{2}-\left(x-x^{*}\right)_{k}^{2}\right)^{4}}\right) \phi_{r}^{\bar y}\,.
	\end{aligned}
	\end{align}
Similar formulas follow for the derivatives with respect to $y$.  Then, we get the following result, whose proof is postponed to the Appendix.

\begin{proposition}\label{propositive}
	Let $\phi_r^{\bar y}$ be the mollifier defined in \eqref{moli} and $(\OX_t,\OY_t)_{0\leq t\leq T}$ be the solution to \eqref{MVeq} up to any time $T>0$. Assume that $\sup_{t\in[0,T]}\EE[\left|x^{*}-X_{\alpha,\beta}(\rho_t^X)\right|+\left|y^{*}-Y_{\beta}(\rho_t^Y,\OX_t)\right|]\leq B$ for some $B>0$. Then there exists some constant $\vartheta>0$ depends only on $d,\lambda,\sigma,r,y^*,\bar c_2$ and $B$ such that
\begin{equation}\label{positive}
	\PP((\OX_t,\OY_t)\in  B_r(x^*)\otimes  B_r({\bar y}))\geq \EE[\phi_r^{\bar y}(\OX_t,\OY_t)]\exp(-\vartheta t)\,.
\end{equation}
\end{proposition}

\begin{thm} \label{theorem:main} Let $(\OX_t,\OY_t)_{t\geq 0}$ be a solution to the equation \eqref{MVeq}.
	Assume that $\TE$ satisfy Assumption \ref{assum2}, and $\lambda,\sigma$ satisfy $2\lambda>\sigma^2$ and $\bar{c}_1\leq \frac{1}{4}\min\left\{\frac{2\lambda-\sigma^2}{16(\lambda+\sigma^2)},\sqrt{\frac{2\lambda-\sigma^2}{16\sigma^2}}\right\}$. Furthermore assume the initial data satisfies 
\begin{equation*} 
 \rho_0^X (\{x:x\in B_{r_0^*}(x^*)\}), \inf_{  \{|\bar y|\leq \bar c_2\}} \rho_0^Y ({y:y\in B_{r_0}(\bar y)}), \inf_{  \{|\bar y|\leq \bar c_2\}} \EE[\phi_{ r_1}^{\bar y}(\OX_0,\OY_0)], \inf_{  \{|\bar y|\leq \bar c_2\}} \EE[\phi_{ r_1^*}^{\bar y}(\OX_0,\OY_0)]>0\quad (*)
 \end{equation*}
 for some $r_0,r_0^*,r_1,r_1^*>0$ to be determined later, and $\mc{V}(0)\geq 2\varepsilon$ for any prescribed accuracy $\varepsilon>0$. Then for $\alpha,\beta$ sufficiently large, there exists some $0<T_*\leq T_\varepsilon$ such that
\begin{align}
	\mc{V}(t)\leq \mc{V}(0)\exp\left(-\frac{2\lambda-\sigma^2}{2}t\right)\mbox{ for }t\in[0,T_\ast)\,,
\end{align}
where $T_\varepsilon:=\frac{2}{2\lambda-\sigma^2}\log(\frac{\mc{V}(0)}{\varepsilon})$. It reaches the prescribed accuracy at time $T_*$, namely $\mc{V}(T_\ast)=\varepsilon$.
\end{thm}
\begin{rmk}
It is obvious to see that the initial preparation $(*)$ may be satisfied for any $r_0,r_0^*,r_1,r_1^*>0$ if the density function of $(X_0,Y_0)$ is continuous and has support in the whole space $\mathbb R^{2d}$. This may be easily satisfied when generating initial samples of  $(X_0,Y_0)$, for instance, with Gaussian distributions.
\end{rmk}
	\begin{proof}[Proof of Theorem \ref{theorem:main}]
		Set
		\begin{equation}
			T_{\alpha,\beta}=\inf\left\{t\geq0\,:\,  \mc{V}(t)\leq \varepsilon \mbox{ or } \EE[|x^*-X_{\alpha,\beta}(\rho_t^X)|+|y^*-Y_\alpha(\rho_t^Y,\OX_t)|]\geq 4C_{\alpha,\beta}(0) \right\}\,,
		\end{equation}
	with 
	\begin{equation}
		C_{\alpha,\beta}(t):=c_3(\mc{V}(t))^\frac{1}{2} \mbox{ and } c_3:=\min\left\{\frac{2\lambda-\sigma^2}{16 
  (\lambda+\sigma^2)},\sqrt{\frac{2\lambda-\sigma^2}{16\sigma^2}}\right\}.
	\end{equation}
Firstly it follows from Proposition \ref{propY} that there exists some $r_0$ such that
{\small 	\begin{align}\label{eq47}
		&\EE[|\bxyn-Y_\beta(\rho_0^Y,\OX_0)|]\leq \frac{(2q_0)^\nu}{\eta}+\frac{\exp\left(-\beta q_0\right)}{\inf_{  \{|\bar y|\leq \bar c_2\}} \rho_0^Y ({y:y\in B_{r_0}(\bar y)})}\int \EE[|y-\bxyn|^2]^{1/2}\rho_0^Y(dy)\nn\\
		\leq& \frac{(2q_0)^\nu}{\eta}+\frac{\exp\left(-\beta q_0\right)}{\inf_{  \{|\bar y|\leq \bar c_2\}} \rho_0^Y ({y:y\in B_{r_0}(\bar y)})}\sqrt{2} \sqrt{\mc{V}^Y(0)}+\frac{\exp\left(-\beta q_0\right)}{\inf_{  \{|\bar y|\leq \bar c_2\}} \rho_0^Y ({y:y\in B_{r_0}(\bar y)})}\sqrt{2}\bar c_1 (\mc{V}^X(0))^{\frac{1}{2}}\nn\\
		\leq& \frac{(2q_0)^\nu}{\eta}+\frac{\exp\left(-\beta q_0\right)}{\inf_{  \{|\bar y|\leq \bar c_2\}} \rho_0^Y ({y:y\in B_{r_0}(\bar y)})}2(1+\bar c_1) (\mc{V}(0))^{\frac{1}{2}}
		\,,
\end{align}}
where in the second inequality we have used the fact that
\begin{align}
|y-\bxyn|^2\leq 2|y-y^*|^2+ 2|y^*-\bxyn|^2
	=2|y-y^*|^2+ 2|\bar y(x^*)-\bxyn|^2
	\leq 2|y-y^*|^2+  2\bar{c}_1^2 |\OX_0-x^*|^2\,.
\end{align}
Here one has constructed that
\begin{equation}\label{264}
	q_0:=\frac{1}{2}\min\left\{(\frac{\eta C_{\alpha,\beta}(0)}{2})^{1/\nu},(\frac{\eta C_{\alpha,\beta}(0)}{C_1^\alpha 4\sqrt{2}})^{1/\nu},\TE_\infty\right\}
\end{equation}
and
\begin{equation}
	r_0=\max\left\{s\in[0,R_0]: \sup_{x\in\RR^d}\max_{y\in B_s(\bar y (x))} |\TE(x,y)-\TE(x,\bar y(x))|\leq q_0\right\}\,.
\end{equation}
By construction, these choices satisfy $r_0\leq R_0$ and $\TE_{r_0}(\OX_0)\leq \frac{\TE_\infty}{2}$ a.s.
Then it holds that
\begin{align}
	&\EE[|y^*-Y_\beta(\rho_0^Y,\OX_0)|]\leq \EE[|y^*-\bxyn|]+\EE[ |\bxyn-Y_\beta(\rho_0^Y,\OX_0)|]\nn\\
	\leq &\frac{c_3}{4}(\mc{V}^X(0))^\frac{1}{2}+\frac{C_{\alpha,\beta}(0)}{2}+\frac{\exp\left(-\beta q_0\right)}{\inf_{  \{|\bar y|\leq \bar c_2\}} \rho_0^Y ({y:y\in B_{r_0}(\bar y)})}2(1+\bar{c}_1) (\mc{V}(0))^{\frac{1}{2}}\nn\\
 \leq& \frac{3C_{\alpha,\beta}(0)}{4}+\frac{\exp\left(-\beta q_0\right)}{\inf_{  \{|\bar y|\leq \bar c_2\}} \rho_0^Y ({y:y\in B_{r_0}(\bar y)})}2(1+\bar{c}_1) (\mc{V}(0))^{\frac{1}{2}}
 \,,
\end{align}
where we have used the fact that $\bar{c}_1\leq \frac{c_3}{4}$. From the definition of $q_0$ in \eqref{264} it follows that $q_0$ is independent of $\beta$ because $C_{\alpha,\beta}(0)$ is independent of $\alpha,\beta$.
 Then one can choose $\beta$ sufficiently large such that
\begin{equation}
\frac{\exp\left(-\beta q_0\right)}{\inf_{  \{|\bar y|\leq \bar c_2\}} \rho_0^Y ({y:y\in B_{r_0}(\bar y)})}2(1+\bar{c}_1)\leq \frac{c_3}{4}\,,
\end{equation}
which leads to
%\begin{equation}
%	\beta_0=\frac{-\log(c_3)+\log (\EE\left[\left|\frac{1}{ \rho_0^Y ({y:y\in B_{r_0}(\bxyn)})} \right|^2\right]^{1/2})+\log(8(1+\bar{c}_1))}{q_0}\,.
%\end{equation}
one has
\begin{equation}
   \EE[|y^*-Y_\beta(\rho_0^Y,\OX_0)|]\leq \frac{3C_{\alpha,\beta}(0)}{4}+\frac{C_{\alpha,\beta}(0)}{4}= C_{\alpha,\beta}(0).%\quad \forall m\in [M].
\end{equation}

We further notice that
{\small 	\begin{align}\label{eq47'}
		&\EE[|\bxyn-Y_\beta(\rho_0^Y,\OX_0)|^2]\leq 2 [\frac{(2q_0)^\nu}{\eta}]^2+2[\frac{\exp\left(-\beta q_0\right)}{\inf_{  \{|\bar y|\leq \bar c_2\}} \rho_0^Y ({y:y\in B_{r_0}(\bar y)})}]^2\int \EE[|y-\bxyn|^2]\rho_0^Y(dy)\nn\\
		\leq& 2[\frac{(2q_0)^\nu}{\eta}]^2+2[\frac{\exp\left(-\beta q_0\right)}{\inf_{  \{|\bar y|\leq \bar c_2\}} \rho_0^Y ({y:y\in B_{r_0}(\bar y)})}]^2 \left(2 \mc V^Y(0)+2\bar c_1^2\mc V^X(0)\right)\nn\\
		\leq& 2[\frac{(2q_0)^\nu}{\eta}]^2+2[\frac{\exp\left(-\beta q_0\right)}{\inf_{  \{|\bar y|\leq \bar c_2\}} \rho_0^Y ({y:y\in B_{r_0}(\bar y)})}]^22(1+\bar c_1^2) \mc V(0)
		\,,
\end{align}}
Similarly it follows from Proposition \ref{propX} that there exists some $r_0^*$ such that
	\begin{align}
	|x^*-X_{\alpha,\beta}(\rho_0^X)|\nn
	\leq&\frac{(2q_0^*)^\nu}{\eta^*}+\frac{\exp\left(-\alpha q_0^*\right)}{ \rho_0^X (\{x:x\in B_{r_0^*}(x^*)\})} \int|x-x^*|\rho_0^X(dx)\nn\\
	&+C_1^\alpha \left(\int_{\RR^{d}}|\bxy-Y_\beta(\rho_0^Y,x)|^2\rho_0^X(dx)\right)^\frac{1}{2}\nn\\
	\leq& \frac{(2q_0^*)^\nu}{\eta^*}+\frac{\exp\left(-\alpha q_0^*\right)}{ \rho_0^X (\{x:x\in B_{r_0^*}(x^*)\})} (\mc{V}^X(0))^{\frac{1}{2}}\nn\\
	&+C_1^\alpha \left(\sqrt 2\frac{(2q_0)^\nu}{\eta}+\frac{\exp\left(-\beta q_0\right)}{\inf_{  \{|\bar y|\leq \bar c_2\}} \rho_0^Y ({y:y\in B_{r_0}(\bar y)})}2(1+\bar{c}_1) (\mc{V}(0))^{\frac{1}{2}}\right)
	\,,
\end{align}
where in the second inequality we have used \eqref{eq47'}.
Here one has constructed that
\begin{equation}
	q_0^*:=\frac{1}{2}\min\left\{(\frac{\eta^* C_{\alpha,\beta}(0)}{4})^{1/\nu},\TE_\infty^*\right\}
\end{equation}
and
\begin{equation}
	r_0^*=\max\left\{s\in[0,R_0]: \sup_{x\in B_r(x^*)}\left|\overline \TE(x)-\overline \TE(x^*)\right|\leq q_0^*\right\}\,.
\end{equation}
By construction, these choices satisfy $r_0^*\leq R_0$ and $\TE_{r_0^*}\leq \frac{\TE_\infty^*}{2}$.
Then it holds that
\begin{align}
	&|x^*-X_{\alpha,\beta}(\rho_0^X)|\leq \frac{C_{\alpha,\beta}(0)}{4}+\frac{\exp\left(-\alpha q_0^*\right)}{ \rho_0^X (\{x:x\in B_{r_0^*}(x^*)\})} (\mc{V}(0))^{\frac{1}{2}}+\frac{C_{\alpha,\beta}(0)}{4}\nn\\
	&+C_1^\alpha  \frac{\exp\left(-\beta q_0\right)}{\inf_{  \{|\bar y|\leq \bar c_2\}} \rho_0^Y ({y:y\in B_{r_0}(\bar y)})}2(1+\bar{c}_1) (\mc{V}(0))^{\frac{1}{2}}
	\,,
\end{align}
where we have used \eqref{264}. Now one can choose $\alpha$ sufficiently large such that
\begin{equation}
	\frac{\exp\left(-\alpha q_0^*\right)}{ \rho_0^X (\{x:x\in B_{r_0^*}(x^*)\})}\leq \frac{c_3}{4}
\end{equation}
and then $\beta$ sufficiently large such that
\begin{equation}
	C_1^\alpha \frac{\exp\left(-\beta q_0\right)}{\inf_{  \{|\bar y|\leq \bar c_2\}} \rho_0^Y ({y:y\in B_{r_0}(\bar y)})}2(1+\bar{c}_1)\leq \frac{c_3}{4}.
\end{equation}
This yields that
\begin{equation}
	|x^*-X_{\alpha,\beta}(\rho_0^X)|\leq C_{\alpha,\beta}(0).%\quad \forall m\in [M].
\end{equation}
So we have
\begin{equation}
	\EE[|x^*-X_{\alpha,\beta}(\rho_0^X)|+|y^*-Y_\alpha(\rho_0^Y,\OX_0)|]\leq 2C_{\alpha,\beta}(0)\,.
\end{equation}
This together with the assumption that $\mc{V}(0)\geq 2\varepsilon$ implies $T_{\alpha,\beta}>0$ .

According to the definition of $T_{\alpha,\beta}$ one has $\mc{V}(t)>\varepsilon$  and  $\EE[|x^*-X_{\alpha,\beta}(\rho_t^X)|+|y^*-Y_\alpha(\rho_t^Y,\OX_t)|]< 4C_{\alpha,\beta}(0)$ for all $t\in [0,T_{\alpha,\beta})$, and at $t=T_{\alpha,\beta}$ it holds that $\mc{V}(T_{\alpha,\beta})=\varepsilon$ or  $\EE[|x^*-X_{\alpha,\beta}(\rho_t^X)|+|y^*-Y_\alpha(\rho_t^Y,\OX_t)|]=4C_{\alpha,\beta}(0)$. Next we will prove that $T_{\alpha,\beta}\leq T_\varepsilon$ and $\mc{V}(t)$ decreases exponentially on $[0,T_{\alpha,\beta})$.

\textbf{Case }$T_{\alpha,\beta}\leq T_\varepsilon$: Similar to \eqref{eq47} there exists some $r_1$ such that
\begin{align}
	&\EE[|\bxynt-Y_\beta(\rho_t^Y,\OX_t)|]\nn\\
\leq& \frac{(2q_1)^\nu}{\eta}+\frac{\exp\left(-\beta q_1\right)}{\inf_{  \{|\bar y|\leq \bar c_2\}} \rho_t^Y ({y:y\in B_{r_1}(\bar y)})}2(1+\bar{c}_1) (\mc{V}(t))^{\frac{1}{2}}
\,,
	\,,
\end{align}
Here one has constructed that
\begin{equation}\label{eqq1}
    q_1:=\frac{1}{2}\min\left\{(\frac{\eta c_3\sqrt{\varepsilon}}{2})^{1/\nu},(\frac{\eta c_3\sqrt{\varepsilon}}{C_1^\alpha4\sqrt{2}})^{1/\nu},\TE_\infty\right\}
\end{equation}
and
\begin{equation}
	r_1=\max\left\{s\in[0,R_0]: \sup_{x\in\RR^d}\max_{y\in B_s(\bar y(x))} |\TE(x,y)-\TE(x,\bar y(x))|\leq q_1\right\}\,.
\end{equation}
Then it holds that
\begin{align}
	\EE[|y^*-Y_\beta(\rho_t^Y,\OX_t)|]\leq  \frac{3C_{\alpha,\beta}(t)}{4}+\frac{\exp\left(-\beta q_1\right)}{\inf_{  \{|\bar y|\leq \bar c_2\}} \rho_t^Y ({y:y\in B_{r_1}(\bar y)})}2(1+\bar{c}_1) (\mc{V}(t))^{\frac{1}{2}}
	\,,
\end{align}
Moreover $\inf_{  \{|\bar y|\leq \bar c_2\}}  \rho_t^Y (\{y:y\in B_{r_1}(\bar y)\})=\inf_{  \{|\bar y|\leq \bar c_2\}}  \PP(\OY_t\in B_{r_1}(\bar y))\geq \inf_{  \{|\bar y|\leq \bar c_2\}}  \EE[\phi_{r_1}^{\bar y}(\OX_0,\OY_0)]\exp(-\vartheta t)$ holds, where $\vartheta$ depends only on $d,\lambda,\sigma,y^*,\bar c_2$ and $B=4C_{\alpha,\beta}(0)$.  
Since  $C_{\alpha,\beta}(0)$ is actually independent of $\alpha,\beta$, we know $\vartheta$ is independent of $\alpha,\beta$.
This concludes that for all $t\in[0,T_{\alpha,\beta})$
\begin{align}\label{eq59}
		&\EE[|y^*-Y_\beta(\rho_t^Y,\OX_t)|]\leq  \frac{3C_{\alpha,\beta}(t)}{4}+\frac{\exp\left(-\beta q_1\right)\exp(\vartheta T_\varepsilon)}{\inf_{  \{|\bar y|\leq \bar c_2\}}  \EE[\phi_{r_1}^{\bar y}(\OX_0,\OY_0)]}2(1+\bar{c}_1) (\mc{V}(t))^{\frac{1}{2}}\leq C_{\alpha,\beta}(t)
		\,,
\end{align}
where we choose $\beta$ sufficiently large.

Similarly  there exists some $r_1^*$ such that
\begin{align}
	&|x^*-X_{\alpha,\beta}(\rho_t^X)|\nn\\
	\leq& \frac{(2q_1^*)^\nu}{\eta^*}+\frac{\exp\left(-\alpha q_1^*\right)}{ \rho_t^X ({x:x\in B_{r_1^*}(x^*)})} (\mc{V}^X(t))^{\frac{1}{2}}\nn\\
	&+C_1^\alpha \left(\sqrt 2\frac{(2q_1)^\nu}{\eta}+\frac{\exp\left(-\beta q_1\right)}{\inf_{  \{|\bar y|\leq \bar c_2\}} \rho_t^Y ({y:y\in B_{r_1}(\bar y)})}2(1+\bar{c}_1) (\mc{V}(t))^{\frac{1}{2}}\right)
	\,.
\end{align}
Here one has constructed that
\begin{equation}
	q_1^*:=\frac{1}{2}\min\left\{(\frac{\eta^* c_3\sqrt{\varepsilon}}{4})^{1/\nu},\TE_\infty^*\right\}
\end{equation}
and
\begin{equation}
	r_1^*=\max\left\{s\in[0,R_0]: \sup_{x\in B_r(x^*)}\left|\overline \TE(x)-\overline \TE(x^*)\right|\leq q_1^*\right\}\,.
\end{equation}
By construction, these choices satisfy $r_1^*\leq R_0$ and $\TE_{r_1^*}\leq \frac{\TE_\infty^*}{2}$.
Then it holds that
\begin{align}
	&|x^*-X_{\alpha,\beta}(\rho_t^X)|\leq \frac{C_{\alpha,\beta}(t)}{4}+\frac{\exp\left(-\alpha q_1^*\right)}{ \rho_t^X ({x:x\in B_{r_1^*}(x^*)})} (\mc{V}(t))^{\frac{1}{2}}+\frac{C_{\alpha,\beta}(t)}{4}\nn\\
	&+C_1^\alpha\frac{\exp\left(-\beta q_1\right)}{\inf_{  \{|\bar y|\leq \bar c_2\}} \rho_t^Y ({y:y\in B_{r_1}(\bar y)})}2(1+\bar{c}_1) (\mc{V}(t))^{\frac{1}{2}}
	\,,
\end{align}
where we have used \eqref{eqq1}. Moreover one has
 $\rho_t^X ({x:x\in B_{r_1^*}(x^*)})=\PP(\OX_t\in B_{r_1^*}(x^*))\geq \EE[\phi_{ r_1^*}^{\bar y}(\OX_0,\OY_0)]\exp(-\vartheta t)$
and
$\inf_{  \{|\bar y|\leq \bar c_2\}}  \rho_t^Y (\{y:y\in B_{r_1}(\bar y)\})\geq \inf_{  \{|\bar y|\leq \bar c_2\}}  \EE[\phi_{r_1}^{\bar y}(\OX_0,\OY_0)]\exp(-\vartheta t)$ according to \eqref{positive}, where $\vartheta$ depends only on $d,\lambda,\sigma,y^*,\bar c_2$ and $B=4C_{\alpha,\beta}(0)$.  
Since  $C_{\alpha,\beta}(0)$ is actually independent of $\alpha,\beta$, we know $\vartheta$ is independent of $\alpha,\beta$.
This concludes that for all $t\in[0,T_{\alpha,\beta})$
\begin{align}
	&|x^*-X_{\alpha,\beta}(\rho_t^X)|\leq \frac{C_{\alpha,\beta}(t)}{4}+\frac{\exp\left(-\alpha q_1^*\right)\exp(\vartheta T_\varepsilon)}{ \EE[\phi_{ r_1^*}^{\bar y}(\OX_0,\OY_0)]} (\mc{V}(t))^{\frac{1}{2}}+\frac{C_{\alpha,\beta}(t)}{4}\nn\\
	&+C_1^\alpha\frac{\exp\left(-\beta q_0\right)\exp(\vartheta T_\varepsilon)}{\inf_{  \{|\bar y|\leq \bar c_2\}}  \EE[\phi_{ r_1}^{\bar y}(\OX_0,\OY_0)]}2(1+\bar{c}_1) (\mc{V}(t))^{\frac{1}{2}}
	\,.
\end{align}
 Now one can choose $\alpha$ sufficiently large such that
\begin{equation}
	\frac{\exp\left(-\alpha q_1^*\right)\exp(\vartheta T_\varepsilon)}{ \EE[\phi_{ r_1^*}^{\bar y}(\OX_0,\OY_0)]}\leq \frac{c_3}{4}
\end{equation}
and then $\beta$ sufficiently large such that
\begin{equation}
	C_1^\alpha\frac{\exp\left(-\beta q_0\right)\exp(\vartheta T_\varepsilon)}{\inf_{  \{|\bar y|\leq \bar c_2\}}  \EE[\phi_{r_1}^{\bar y}(\OX_0,\OY_0)]} 2(1+\bar{c}_1) \leq \frac{c_3}{4}.
\end{equation}
This yields that
\begin{equation}
	|x^*-X_{\alpha,\beta}(\rho_0^X)|\leq C_{\alpha,\beta}(t).%\quad \forall m\in [M].
\end{equation}
So we have
\begin{equation}\label{290}
	\EE[|x^*-X_{\alpha,\beta}(\rho_t^X)|+|y^*-Y_\alpha(\rho_t^Y,\OX_t)|]\leq 2C_{\alpha,\beta}(t).
\end{equation}

		Let us recall the upper bound for the time derivative of $\mc{V}(t)$ given in Lemma \ref{lemV}:
\begin{align}
	\frac{d \mc{V}(t)}{dt}&\leq -(2\lambda-\sigma^2)\mc{V}(t)+2\sqrt{2}(\lambda+\sigma^2)(\mc{V}(t))^{\frac{1}{2}}(|x^*-X_{\alpha,\beta}(\rho_t^X)|+\EE[|y^*-Y_\beta(\rho_t^Y,\OX_t)|])\nn\\
	&\quad +\sigma^2 (|x^*-X_{\alpha,\beta}(\rho_t^X)|^2+\EE[|y^*-Y_\beta(\rho_t^Y,\OX_t)|^2])\,.
\end{align}
Then from \eqref{290} one can deduce that
\begin{equation}
		\frac{d \mc{V}(t)}{dt}\leq-\frac{2\lambda-\sigma^2}{2}\mc{V}(t)\,,
\end{equation}
which using Gr{\"o}nwall's inequality leads to
\begin{equation}
\mc{V}(t)\leq \mc{V}(0)\exp\left(-\frac{2\lambda-\sigma^2}{2}t\right)\mbox{ for }t\in[0,T_{\alpha,\beta})\,.
\end{equation}
This implies
\begin{equation}
\EE[|x^*-X_{\alpha,\beta}(\rho_{T_{\alpha,\beta}}^X)|+|y^*-Y_\alpha(\rho_{T_{\alpha,\beta}}^Y,\OX_{T_{\alpha,\beta}})|]\leq 2C_{\alpha,\beta}({T_{\alpha,\beta}})=2c_3(\mc{V}(T_{\alpha,\beta}))^\frac{1}{2}\leq  2C_\alpha(0)\,,
\end{equation}
Then according to the definition of $T_{\alpha,\beta}$ we must have $\mc{V}(T_{\alpha,\beta})\leq \varepsilon$.
	
\textbf{Case} $T_\varepsilon<T_{\alpha,\beta}$: By the definition of $T_{\alpha,\beta}$ we know that $\mc{V}(t)>\varepsilon$ and $\EE[|x^*-X_{\alpha,\beta}(\rho_t^X)|+|y^*-Y_\alpha(\rho_t^Y,\OX_t)|]< 4C_\alpha(0)$ for all $t\in[0,T_\varepsilon]$. Then following the same arguments as the in the first case, one concludes 
\begin{align}
		\mc{V}(t)\leq \mc{V}(0)\exp\left(-\frac{2\lambda-\sigma^2}{2}t\right)\mbox{ for }t\in[0,T_\varepsilon].
\end{align}
Using this estimate, the fact that $T_\varepsilon=\frac{2}{2\lambda-\sigma^2}\log(\frac{\mc{V}(0)}{\varepsilon})$ implies $	\mc{V}(T_\varepsilon)= \varepsilon$, which is a contradiction. Thus this case can never happen.
\end{proof}

\section{Numerical tests}
\label{sec:numerics}

We now validate the performance of the proposed CBO dynamics \eqref{particle} against nonconvex-nonconcave benchmark problems. We are particularly interested in numerically verifying if the convergence guarantees obtained for the mean-field model \eqref{MVeq} still holds when a finite number $N$ of particles is used. Also, we investigate how the different algorithm's parameters influence its convergence properties.

To obtain the actual numerical algorithm, we discretize the SDE particle system via explicit Euler--Maruyama scheme \cite{euler_scheme_1} which has an accuracy of order $1/2$. Let $\Delta t \in (0,1)$ be the stepsize, and the index $k$ denote the time step. Starting from a collection of $2N$ particles $\{(X_{(0)}^{i},Y_{(0)}^{i})\}_{i\in [N]}\subset \R^{d}\times \R^{d}$ the Euler--Maruyama discretization to \eqref{particle} reads
\begin{subequations}\label{discrete_particle}
	\begin{numcases}{}
		X_{(k+1)}^i= X_{(k)}^i -\lambda\Delta t(X_{(k)}^i-X_{\alpha,\beta}(\rho_{(k)}^{N,X}))+\sigma \sqrt{\Delta t} D(X_{(k)}^i-	X_{\alpha,\beta}(\rho_{(k)}^{N,X}))B_{(k)}^{X,i}\,, \quad i\in[N]\,,\label{eqX_discr}\\
		Y_{(k+1)}^i= Y_{(k)}^i -\lambda\Delta t(Y_{(k)}^i-Y_\beta^i(\rho_{(k)}^{N,Y},X_{(k)}^i))+\sigma \sqrt{\Delta t} D(Y_{(k)}^i-Y_\beta(\rho_{(k)}^{N,Y},X_{(k)}^i))B_{(k)}^{Y,i}\,,\quad  i\in[N]\label{eqY_discr}
	\end{numcases}
\end{subequations}
where $\rho_{(k)}^{N,X},\rho_{(k)}^{N,Y}$ are the empirical measures associated with the two particle populations, and $B_{(k)}^{X,i}, B_{(k)}^{Y,i}$ for all $i \in [N]$ are vectors independently sampled from the standard normal distributions $\mathcal{N}(0,I_{d_1})$ and $\mathcal{N}(0,I_{d_2})$ respectively.

\subsection{Results for nonconvex-nonconcave benchmark problems}

We consider the following challenging nonconvex-nonconcave problems which are known to violate regularity conditions such as Polyak--Łojasiewicz or Kurdyka--Łojasiewicz conditions 
\cite{zheng2023universal}.
\begin{itemize}
\item Problem A: "Bilinearly-coupled minimax" \cite{grimmer2023prox}:
\begin{align*}
\TE(x,y) &= f(x) + Axy -   f(y) \\
&\textup{with} \; f(z) =  (z+1)(z-1)(z+3)(z-3), \quad A = 10
\end{align*}
and $\mc{X} = \mc{Y} = [-4,4]$. While the unique stationary point for gradient-based dynamics is $(0,0)$, the actual problem solutions in the sense of Definition \ref{defsolminmax} are given by
\[(x^*,y_\pm^*) \approx (0,\pm 2.24)\,. \]
\item Problem B: "Forsaken" \cite{hsieh2021limits}
\begin{align*}
\TE(x,y) &= x(y-0.45) + \phi(x) - \phi(y) \\
&\textup{with} \; \phi(z) = z^2/4 - z^4/2 + z^6/6
\end{align*}
and $\mc{X} = \mc{Y} = [-1.5,1.5]$. Interestingly, while the optimal solution is considered to be the point $(x^\star,y^\star)\approx (0.08,0.4)$ \cites{hsieh2021limits,zheng2023universal}, the true minmax global solutions in the sense of Definition \ref{defsolminmax} are 
\[(x^*,y_\pm^*) \approx (0,\pm 1.31)\,. \]
\item Problem C: "Sixth-order polynomial" \cites{wang2020ridge}
\begin{align*}
\TE(x,y) &= (4x^2 - (y-3x + 0.05x^3)^2 - 0.1 y^4) \exp(-0.01(x^2 + y^2))\,
\end{align*}
with $\mathcal{X} = \mathcal{Y} = [-2,2]$ and unique solution
\[(x^*, y^*) = (0,0)\,.\]
\end{itemize}
The problems are all one-dimensional in both variables, that is, $d_1 = d_2 = 1$. To ensure the CBO dynamics remains in the prescribed search spaces, we project the particles to $\mc{X},\mc{Y}$ after each time step.

Figure \ref{fig:single_test} illustrate the algorithm dynamics for a single run. On the left pictures, we show the initialization of $N=25$ coupled particles $\{(X^i_{(0)},Y^i_{(0)})\}_{i\in{[N]}}$, and the evolution of the mean values $(m_{(k)}^x, m_{(k)}^y) = (1/N)\sum_{i\in [N]} (X_{(k)}^i, Y_{(k)}^i)$. We note how, in all the problems considered, the algorithm is able to converge towards a global minmax solution. 
On centre and on the right plots of Figure \ref{fig:single_test} we can see the particles trajectories over the spaces $\mc{X}$ and $\mc{Y}$ separately. For Problems A and B particles in the $\mc{X}$ space reaches a consensus over the solution $x^* = 0$ faster with respect to the ones in $\mc{Y}$ space, see Figures \ref{fig:singleA} and \ref{fig:singleB}. This could be justified by the fact the function $\max_y \TE(x^*,y)$ attains two maximal points. We note that, ultimately, particles concentrate around a single solution, creating full consensus around it.
For Problem $C$ we note a more oscillatory behavior around the global solution, for particles both in the $\mc{X}$ and $\mc{Y}$ space, but consensus is eventually reached if we let the particles evolve for a sufficiently large time horizon.

\begin{figure}
\centering
\subfloat[Problem A \label{fig:singleA}]{\includegraphics[width = 1\linewidth]{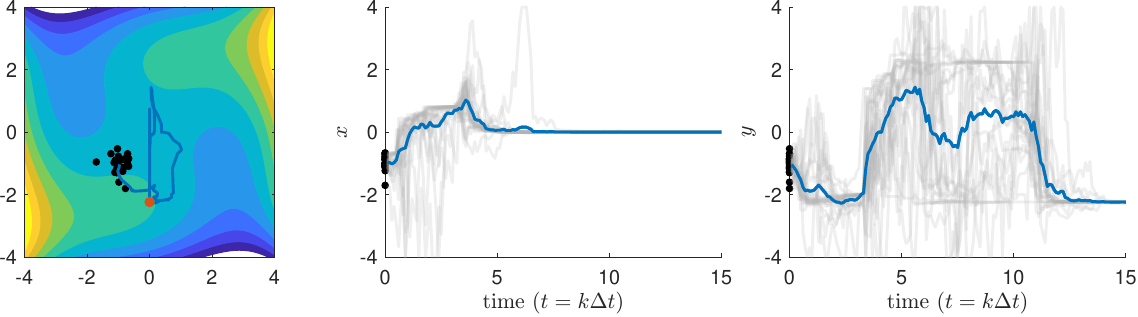}}\\ \bigskip
\subfloat[Problem B\label{fig:singleB}]{\includegraphics[width = 1\linewidth]{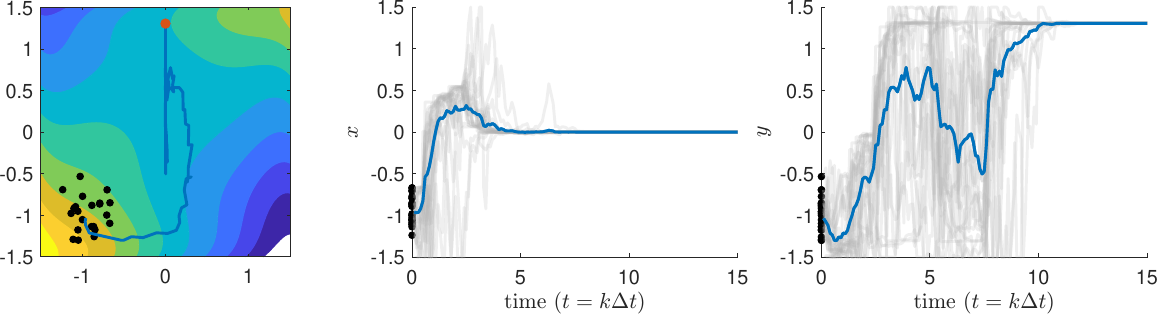}}\\ \bigskip
%\subfloat[Problem C\label{fig:singleC}]{\includegraphics[width = 1\linewidth]{fig_illustrative_5}}
\subfloat[Problem C\label{fig:singleC}]{\includegraphics[trim = 1.6cm 12cm 0cm 12cm,clip,width = 1\linewidth]{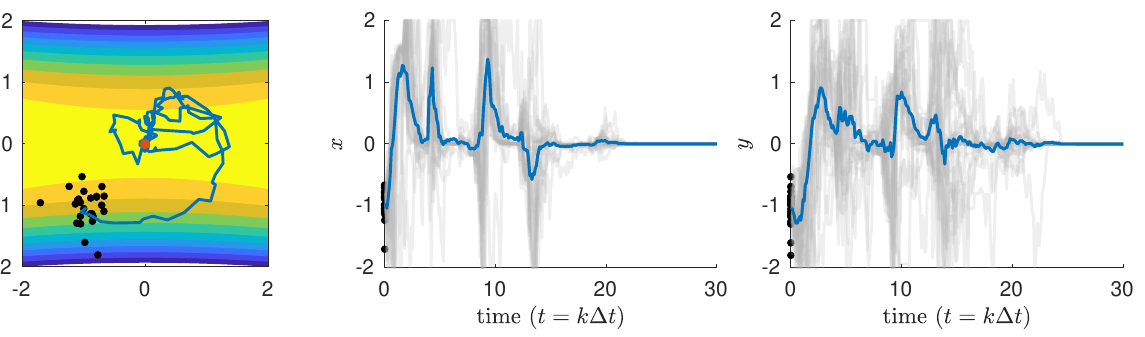}}
\caption{Illustration of a single run fo the algorithm dynamics \eqref{discrete_particle} for the three benchmark problem considered. Black dots represent the particles initialization, while the blue lines the evolution of their (unweighted) means $m^x_{(k)}, m^y_{(k)}$. On the center and on the right, we also display in grey the trajectories of the $N=25$ particles in the $\mc{X}$ and $\mc{Y}$ space, respectively. Algorithm's parameters are set to $\Delta t = 0.1, \lambda = 1, \sigma = 1.5, \alpha = \beta = 10^4$, and the time-horizon is set to $T = 15$ (Problems A and B), $T = 30$ (problem C).}
\label{fig:single_test}
\end{figure}

In the previous experiment we used the same parameters $N = 25, \Delta t = 0.1, \lambda = 1, \sigma = 1.5, \alpha = \beta = 10^4$ of CBO dynamics \eqref{discrete_particle} in all test problems. We now test the algorithm by, separately, tuning a parameter's value and verify how the performance of the algorithm changes over 100 tests. We let the particles evolve over a time horizon $T = 50$ and keep a fixes step size of $\Delta t=0.1$. We consider the computed solutions at the final step $K = T/\Delta t$ to be the best particles among the ensemble, that is, a couple $(X^{i^*},Y^{j^*})$ satisfying
\[\TE(X^{i^*},Y^{j^*}) =
\min_{i\in[N]}\max_{j \in [N]}\TE(X_{(K)}^{i},Y^{j}_{(K)}).
 \ \]
The error is computed as the squared norm from the (closest) global minimax point
\[| (X^{i^*},Y^{j^*}) - (x^*,y^*)|^2.\]
The particles positions are randomly initialized at the border of the search space $\mathcal{X} \times \mathcal{Y}$ to avoid the possibility of initializing particles close the global solutions.

We recall that the proposed convergence analysis for the particles dynamics has been carried out by relying on the mean-field approximation \eqref{MVeq} which is expected to be exact as $N \to \infty$. Therefore, the finite-particle system introduces an additional approximation error to the convergence established in Theorem \ref{theorem:main}. From Figure \ref{fig:N} we note that the error consistently decreases as the number of particles per population $N$ increases. The improvement is substantial up until approximately $N=120$ particles, and it reaches a plateau when more than $N = 10^2$ particles per population are used. It also interesting to notice the interquantile range decreases as $N$ increases. This is consistent with the mean-field limit assumption, which assume that the particles' empirical measure approaches a deterministic distribution as the number of particles $N$ approaches infinity.

Other parameters that play a central role in the convergence result Theorem \ref{theorem:main} are $\alpha>0$ and $\beta>0$ which enter the dynamics via the computation of the consensus points \eqref{YaN} and \eqref{XaN}. In particular, we recall from the quantitative Laplace principle Propositions \ref{propY} and \ref{propX} that they both need to be large, in relation to the objective function growth rates around the solutions. For Problems A and B, using larger values of $\alpha, \beta$ indeed increases the algorithm accuracy, while for Problem C, values larger than $10^5$ lead to poor results, see Figure \ref{fig:alpha}. We conjecture this may be due to the well-know trade off for CBO algorithm, between the parameter $\alpha$ (and $\beta$ in this case) and the accuracy in the mean-field approximation. Specifically, as $\alpha,\beta$
increases, the particle system is expected to diverge form the mean-field model, with possible lose of the theoretical convergence guarantees \cites{fornasier2020consensus, fornasier2021consensus1,gerber2023mean}.

Next, we show in Figure \ref{fig:sigma} how the algorithm's accuracy depends on the strength $\sigma>0$ of the stochastic component. This is a crucial parameter as low values of $\sigma$ do not provide the particles with a sufficient exploratory behavior, while large values might prevent them to create consensus. For the benchmark problem considered, best performance is reached when $\sigma \in [1,2]$. We also note in Figure \ref{fig:sigma} that large values of $\sigma$ also increases the interquantile range.

Lastly, we consider the case where the two particle populations, in the $\mc{X}$ and $\mc{Y}$ search space, evolve at different time scales. Numerically, this means using in the particles update \eqref{discrete_particle} two different time steps $\Delta t^x$ and $\Delta t^y$. We control the relative speed via an additional parameter $\epsilon >0$ such that 
\[\epsilon = \frac{\Delta t^x}{\Delta t^y}.
\] 
Figure \ref{fig:epsilon} collects the result for $\Delta t^y = \Delta t = 0.1$ fixed, and different values of $\epsilon$. In Problems A and B, the best accuracy is reached for $\epsilon \in [0.1,1]$ (that is, a slower dynamics in $\mc{X}$ space), while for Problem C the optimal performance is reached for $\epsilon = 1$ (that is, same time scale for both dynamics).

\begin{figure}
\centering
\subfloat[Varying number $N$ of particles per population \label{fig:N}]{\includegraphics[width = 0.45\linewidth]{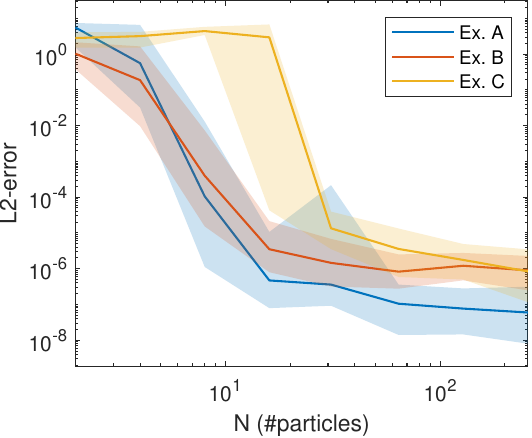}} \hfill
\subfloat[Varying Laplace parameters $\alpha, \beta$ \label{fig:alpha}]{\includegraphics[width = 0.45\linewidth]{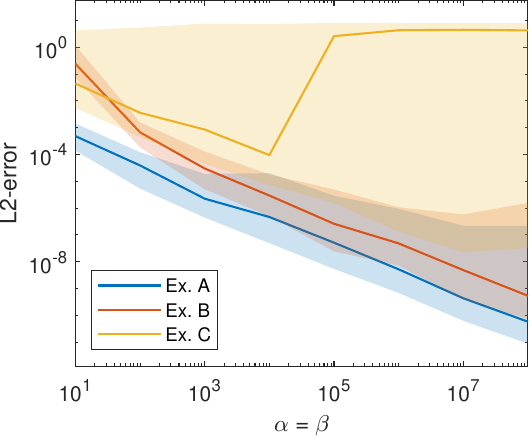}}\\
\subfloat[Varying stochastic strength $\sigma>0$ \label{fig:sigma}]{\includegraphics[width = 0.45\linewidth]{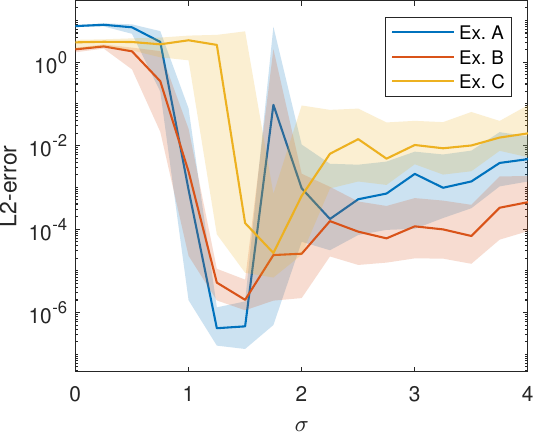}}\hfill
\subfloat[Varying time-scale $\epsilon$, $ \Delta t^x = \epsilon \Delta t^y$ with $\Delta t^y = \Delta t$ \label{fig:epsilon}]{\includegraphics[width = 0.45\linewidth]{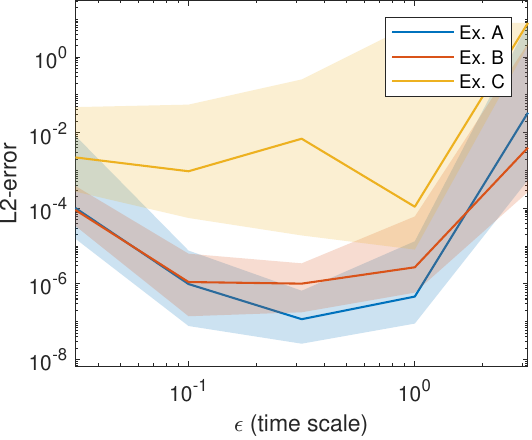}}\hfill
\caption{Accuracy reached for the three benchmark problems considered when varying a parameter of the dynamics. The reference set of parameters is given by $N  = 20, \Delta t = 0.1, \lambda = 1, \sigma = 1.5, \alpha = \beta = 10^4$. The plots illustrate the median value and the quantile range $[0.2, 0.8]$ over 100 runs.
}
\label{fig:param_test}
\end{figure}

%\section{Conclusions}
%\label{sec:conclusions}
%\blue{Add it (G) -- not sure about adding it at the momemnt. Let's see}

\section{Acknowledgements}
The work of G.B. is funded by the Deutsche Forschungsgemeinschaft (DFG, German Research Foundation) through 320021702/GRK2326 ``Energy, Entropy, and Dissipative Dynamics (EDDy)''. The Department of Mathematics and Scientific Computing at the University of Graz, with which H.H. are affiliated, is a member of NAWI Graz (https://nawigraz.at/en). J.Q. is partially supported by the National Science and Engineering Research Council of Canada (NSERC) and by the 2023-2024 PIMS-Europe Fellowship.

\section*{Appendix}
\label{sec:appendix}
\begin{proof}[Proof of Proposition \ref{propositive}]
	For simplicity we drop the sup-index ${\bar y}$ for $\phi_r^{\bar y}$. It is easy to compute
	\begin{align}
		\frac{d \EE[\phi_r(\OX_t,\OY_t)]}{dt}&=\sum_{k=1}^{d}\EE[T_{1,k}^x(\OX_t,\OY_t)+T_{2,k}^x(\OX_t,\OY_t)]\nn\\
		&+\sum_{k=1}^{d}\EE[T_{1,k}^y(\OX_t,\OY_t)+T_{2,k}^y(\OX_t,\OY_t)]
	\end{align}
	with
	\begin{equation}
		T_{1,k}^x:=-\lambda(\OX_t-X_{\alpha,\beta}(\rho_t^X))_{k}\partial_{x_k}\phi_r\mbox{ or }T_{2,k}^x:=\frac{\sigma^2}{2}(\OX_t-X_{\alpha,\beta}(\rho_t^X))_{k}^2\partial_{x_kx_k}^2\phi_r
	\end{equation}
	and
	\begin{equation}
		T_{1,k}^y:=-\lambda(\OY_t-Y_\beta(\rho_t^Y,\OX_t))_{k}\partial_{y_k}\phi_r\mbox{ and }T_{2,k}^y:=\frac{\sigma^2}{2}(\OY_t-Y_\beta(\rho_t^Y,\OX_t))_{k}^2\partial_{y_ky_k}^2\phi_r\,.
	\end{equation}
	Let us define the open $\ell_\infty$-ball $\Omega_r:=\{(x,y)\in\R^{2d}:\, (x,y)\in B_r(x^*)\otimes B_r({\bar y}) \}$, and for each $k=1,\dots,d$ the subsets
	\begin{equation}
		K_{1,k}^x:=\{(x,y)\in \R^{2d}:\,|(x-x^*)_k|>\sqrt{c}r\}
	\end{equation}
	\begin{equation}
		K_{1,k}^y:=\{(x,y)\in \R^{2d}:|(y-{\bar y})_k|>\sqrt{c}r\}
	\end{equation}
	and
	\begin{align}
		K_{2,k}^x:=\bigg\{(x,y) \in \mathbb{R}^{2d}:&-\lambda(x-X_{\alpha,\beta}(\rho_t^X))_{k}(x-x^{*})_{k}(r^{2}-(x-x^{*})_{k}^{2})^{2} \\
		>&\tilde{c} r^{2} \frac{\sigma^{2}}{2}(x-X_{\alpha,\beta}(\rho_t^X))_{k}^{2}(x-x^{*})_{k}^{2}
		\bigg\}
	\end{align}
	\begin{align}
		K_{2,k}^y:=\bigg\{(x,y) \in \mathbb{R}^{2d}:
		&-\lambda(y-Y_\beta(\rho_t^Y,x))_{k}(y-{\bar y})_{k}(r^{2}-(y-{\bar y})_{k}^{2})^{2} \\
		>&\tilde{c} r^{2} \frac{\sigma^{2}}{2}(y-Y_\beta(\rho_t^Y,x))_{k}^{2}(y-{\bar y})_{k}^{2}
		\bigg\}\,,
	\end{align}
	where $\tilde c:=2c-1\in (0,1)$. For fixed $k$ we decompose $\Omega_r$ according to
	\begin{align}
		\Omega_{r}=\left((K_{1,k}^x)^{c} \cap \Omega_{r}\right) \cup\left((K_{1,k}^x)^c \cap (K_{2,k}^x)^{c} \cap \Omega_{r}\right) \cup\left(K_{1,k}^x \cap K_{2,k}^x \cap \Omega_{r}\right)
	\end{align}
	Then in the following we consider each of these subsets separately.
	
	\textbf{Subset} $(K_{1,k}^x)^c\cap \Omega_r:$ For each $(\OX_t,\OY_t)\in (K_{1,k}^x)^c$, we have $|(\OX_t-x^*)|\leq \sqrt{c}r$. To estimate $\EE[T_{1,k}^x]$, we use the expression of $\partial_{x_k}\phi_r$ and get
	\begin{align}
		\EE[T_{1,k}^x] &=2 r^{2} \lambda\EE\left[\left(\OX_t-X_{\alpha,\beta}(\rho_t^X)\right)_{k} \frac{\left(\OX_t-x^{*}\right)_{k}}{\left(r^{2}-\left(\OX_t-x^{*}\right)_{k}^{2}\right)^{2}} \right]\phi_{r} \\
		& \geq-2 r^{2} \lambda \EE\left[ \frac{\left|\left(\OX_t-X_{\alpha,\beta}(\rho_t^X)\right)_{k}\right|\left|\left(\OX_t-x^{*}\right)_{k}\right|}{\left(r^{2}-\left(\OX_t-x^{*}\right)_{k}^{2}\right)^{2}}\right] \phi_{r} \geq-\frac{2 \lambda(\sqrt{c} r+B) \sqrt{c}}{(1-c)^{2} r} \phi_{r} \\
		&=:-q_{1}^x \phi_{r}\,,
	\end{align}
	where in the last inequality we have used 
	\begin{equation}\label{esforx}
		\EE[\left|\left(\OX_t-X_{\alpha,\beta}(\rho_t^X)\right)_{k}\right|] \leq \EE[\left|\left(\OX_t-x^{*}\right)_{k}\right|]+\EE[\left|\left(x^{*}-X_{\alpha,\beta}(\rho_t^X)\right)_{k}\right|] \leq \sqrt{c} r+B\,.
	\end{equation}
	Similarly for $\EE[T_{2,k}^x]$ one obtains
	\begin{align}
		\EE[T_{2,k}^x] &=\sigma^{2} r^{2}\EE\left[\left(\OX_t-X_{\alpha,\beta}(\rho_t^X)\right)_{k}^{2} \frac{2\left(2\left(\OX_t-x^{*}\right)_{k}^{2}-r^{2}\right)\left(\OX_t-x^{*}\right)_{k}^{2}-\left(r^{2}-\left(\OX_t-x^{*}\right)_{k}^{2}\right)^{2}}{\left(r^{2}-\left(\OX_t-x^{*}\right)_{k}^{2}\right)^{4}} \right]\phi_{r} \\
		& \geq-\frac{2 \sigma^{2}\left(c r^{2}+B^{2}\right)(2 c+1)}{(1-c)^{4} r^{2}} \phi_{r}=:-q_{2}^x \phi_{r}\,.
	\end{align}

	\textbf{Subset} $K_{1,k}^x\cap (K_{2,k}^x)^c\cap \Omega_r:$  As  $(\OX_t,\OY_t)\in K_{1,k}^x$ we have $|(\OX_t-x^*)_k|>\sqrt{c}r$. We observe that $T_{1,k}^x+T_{2,k}^x\geq 0$ for all $\OX_t$ satisfying
	\begin{align}\label{positve}
		&\left(-\lambda\left(\OX_t-X_{\alpha,\beta}(\rho_t^X)\right)_{k}\left(\OX_t-x^{*}\right)_{k}+\frac{\sigma^{2}}{2}\left(\OX_t-X_{\alpha,\beta}(\rho_t^X)\right)_{k}^{2}\right)\left(r^{2}-\left(\OX_t-x^{*}\right)_{k}^{2}\right)^{2} \nn\\
		\leq &\sigma^{2}\left(\OX_t-X_{\alpha,\beta}(\rho_t^X)\right)_{k}^{2}\left(2\left(\OX_t-x^{*}\right)_{k}^{2}-r^{2}\right)\left(\OX_t-x^{*}\right)_{k}^{2}\,.
	\end{align}
	Actually, this can be verified by first showing that
	\begin{align}
		&-\lambda\left(\OX_t-X_{\alpha,\beta}(\rho_t^X)\right)_{k}\left(\OX_t-x^{*}\right)_{k}\left(r^{2}-\left(\OX_t-x^{*}\right)_{k}^{2}\right)^{2} \leq \tilde{c} r^{2} \frac{\sigma^{2}}{2}\left(\OX_t-X_{\alpha,\beta}(\rho_t^X)\right)_{k}^{2}\left(\OX_t-x^{*}\right)_{k}^{2} \\
		&\quad=(2 c-1) r^{2} \frac{\sigma^{2}}{2}\left(\OX_t-X_{\alpha,\beta}(\rho_t^X)\right)_{k}^{2}\left(\OX_t-x^{*}\right)_{k}^{2} \leq\left(2\left(\OX_t-x^{*}\right)_{k}^{2}-r^{2}\right) \frac{\sigma^{2}}{2}\left(\OX_t-X_{\alpha,\beta}(\rho_t^X)\right)_{k}^{2}\left(\OX_t-x^{*}\right)_{k}^{2}\,,
	\end{align}
	where we have used the fact that $(\OX_t,\OY_t)\in K_{1,k}^x\cap (K_{2,k}^x)^c$ and $\tilde c=2c-1$. One also notice that
	\begin{align}
		\frac{\sigma^{2}}{2} &\left(\OX_t-X_{\alpha,\beta}(\rho_t^X)\right)_{k}^{2}\left(r^{2}-\left(\OX_t-x^{*}\right)_{k}^{2}\right)^{2} \leq \frac{\sigma^{2}}{2}\left(\OX_t-X_{\alpha,\beta}(\rho_t^X)\right)_{k}^{2}(1-c)^{2} r^{4}\nn \\
		& \leq \frac{\sigma^{2}}{2}\left(\OX_t-X_{\alpha,\beta}(\rho_t^X)\right)_{k}^{2}(2 c-1) r^{2} c r^{2} \leq \frac{\sigma^{2}}{2}\left(\OX_t-X_{\alpha,\beta}(\rho_t^X)\right)_{k}^{2}\left(2\left(\OX_t-x^{*}\right)_{k}^{2}-r^{2}\right)\left(\OX_t-x^{*}\right)_{k}^{2} 
	\end{align}
	by using $(1-c)^2\leq(2c-1)c$. Hence \eqref{positve} holds and we have $T_{1,k}^x+T_{2,k}^x\geq 0$.

	\textbf{Subset} $K_{1,k}^x\cap K_{2,k}^x\cap \Omega_r:$ Notice that when $(\OX_t)_k=(X_{\alpha,\beta}(\rho_t^X))_k$, we have $T_{1,k}^x=T_{2,k}^x=0$, so in this case there is nothing to prove. If $(X_t)_k\neq(X_{\alpha,\beta}(\rho_t^X))_k$, or $\sigma^2(X_t-X_\alpha(\rho_t^X))_k^2>0$ $(\sigma>0)$, we exploit
	$(\OX_t,\OY_t)\in K_{2,k}^x$ to get
	\begin{align}
		\frac{\left(\OX_t-X_{\alpha,\beta}(\rho_t^X)\right)_{k}\left(\OX_t-x^{*}\right)_{k}}{\left(r^{2}-\left(\OX_t-x^{*}\right)_{k}^{2}\right)^{2}} & \geq \frac{-\left|\left(\OX_t-X_{\alpha,\beta}(\rho_t^X)\right)_{k}\right|\left|\left(\OX_t-x^{*}\right)_{k}\right|}{\left(r^{2}-\left(\OX_t-x^{*}\right)_{k}^{2}\right)^{2}} \\
		&>\frac{2 \lambda\left(\OX_t-X_{\alpha,\beta}(\rho_t^X)\right)_{k}\left(\OX_t-x^{*}\right)_{k}}{\tilde{c} r^{2} \sigma^{2}\left|\left(\OX_t-X_{\alpha,\beta}(\rho_t^X)\right)_{k}\right|\left|\left(\OX_t-x^{*}\right)_{k}\right|} \geq-\frac{2 \lambda}{\tilde{c} r^{2} \sigma^{2}} .
	\end{align}
	Using this, $T_{1,k}^x$ can be bounded from below
	\begin{equation}
		T_{1,k}^x=2 r^{2} \lambda\left(\OX_t-X_{\alpha,\beta}(\rho_t^X)\right)_{k} \frac{\left(\OX_t-x^{*}\right)_{k}}{\left(r^{2}-\left(\OX_t-x^{*}\right)_{k}^{2}\right)^{2}} \phi_{r} \geq-\frac{4 \lambda^{2}}{\tilde{c} \sigma^{2}} \phi_{r}=:-q_{3}^x \phi_{r}\,.
	\end{equation}
	Moreover since $(\OX_t,\OY_t)\in K_{1,k}^x$ and $2(2c-1)c\geq(1-c)^2$ implied by the assumption, one has
	\begin{equation}
		2\left(2\left(\OX_t-x^{*}\right)_{k}^{2}-r^{2}\right)\left(\OX_t-x^{*}\right)_{k}^{2} \geq\left(r^{2}-\left(\OX_t-x^{*}\right)_{k}^{2}\right)^{2}\,,
	\end{equation}
	which yields that $T_{2,k}^x\geq 0$. 
	
	\textbf{Concluding the proof:} Collecting estimates from above and similar estimates for $\EE[T_{1,k}^y+T_{2,k}^y]$ under the decomposition of $\Omega_r=\left((K_{1,k}^y)^{c} \cap \Omega_{r}\right) \cup\left((K_{1,k}^y)^c \cap (K_{2,k}^y)^{c} \cap \Omega_{r}\right) \cup\left(K_{1,k}^y \cap K_{2,k}^y \cap \Omega_{r}\right)$, we get
	\begin{align}
		&\frac{d \EE[\phi_r(\OX_t,\OY_t)]}{dt}=\sum_{k=1}^{d}\EE[T_{1,k}^x+T_{2,k}^x]+\sum_{k=1}^{d}\EE[T_{1,k}^y+T_{2,k}^y]\nn\\
		=&\sum_{z=x,y}\sum_{k=1}^{d}\bigg(\EE[(T_{1,k}^z+T_{2,k}^z)\textbf{I}_{K_{1,k}^z\cap K_{2,k}^z\cap \Omega_r}]+\EE[(T_{1,k}^z+T_{2,k}^z)\textbf{I}_{K_{1,k}^z\cap (K_{2,k}^z)^c\cap \Omega_r}]\nn\\
		&+\EE[(T_{1,k}^z+T_{2,k}^z)\textbf{I}_{(K_{1,k}^z)^c\cap \Omega_r}]\bigg)\nn\\
		\geq&-d(\max\{q_1^x+q_2^x,q_3^x\}+\max\{q_1^y+q_2^y,q_3^y\})\EE[\phi_r]=:-\vartheta\EE[\phi_r]\,.
	\end{align}
	An application of Gronwall's inequality concludes that
	\begin{equation}
		\PP((\OX_t,\OY_t)\in  B_r(x^*)\times B_r({\bar y}))\geq \EE[\phi_r(\OX_t,\OY_t)]\geq \EE[\phi_r(\OX_0,\OY_0)]\exp(-\vartheta t)\,.
	\end{equation}
	We remark here that for the estimate of $\EE[T_{1,k}^x]$ in the set of $(K_{1,k}^y)^{c}\cap \Omega_r$, one replace \eqref{esforx} by 
		\begin{equation}
		\EE[\left|\left(\OY_t-Y_\beta(\rho_t^Y,\OX_t)\right)_{k}\right|] \leq \EE[\left|\left(\OY_t-\bar y\right)_{k}\right|]+|(y^*-\bar y)_k|+\EE[\left|\left(y^{*}-Y_\beta(\rho_t^Y,\OX_t)\right)_{k}\right|] \leq \sqrt{c} r+\bar c_2+|y^*|+B\,.
	\end{equation}
\end{proof}

%%%%%%%%%%%%%%%%%%%%%%%%%%%%%%%%%%%%%%
\bibliographystyle{amsxport}
%\bibliography{minmax,multi}
\bibliography{minmax}
%%%%%%%%%%%%%%%%%%%%%%%%%%%%%%%%%%%%%%

%\bibliography{bounded}
%\bibliographystyle{abbrv}

\end{document}